\documentclass[12pt]{article}

\usepackage{mathptmx}       
\usepackage{helvet}         
\usepackage{courier}        
\usepackage{type1cm}        
\usepackage{graphicx}        
\usepackage{multicol}        
\usepackage[bottom]{footmisc}

\usepackage[utf8]{inputenc}
\usepackage[T1]{fontenc}
\usepackage{amsmath}
\usepackage{amsthm}
\usepackage{amssymb}
\usepackage[english]{babel}

\allowdisplaybreaks[3]

\usepackage{enumitem}  
\setlist{itemsep=.5em} 
\setlist[itemize]{leftmargin=1.2cm}

\usepackage{array}
\usepackage{amscd}
\usepackage{amsfonts}
\usepackage{amsbsy}
\usepackage{mathbbol}
\usepackage{amsmath}
\usepackage{amssymb}
\usepackage{commath}
\usepackage{amstext}
\usepackage{latexsym}
\usepackage{longtable}
\usepackage{verbatim}
\usepackage{comment}
\usepackage{mathtools}
\usepackage{cite}
\usepackage{epstopdf}
\usepackage[english]{babel}
\usepackage{booktabs}
\usepackage{multirow}
\usepackage{chapterbib}

\usepackage{mathrsfs}

\newtheorem{theorem}{Theorem}[section]
\newtheorem{lemma}[theorem]{Lemma}
\newtheorem{corollary}[theorem]{Corollary}
\newtheorem{proposition}[theorem]{Proposition}
\newtheorem{definition}[theorem]{Definition}
\newtheorem{example}[theorem]{Example}

\newtheorem{remark}[theorem]{Remark}

\begin{document}

\title{Multiplicative $n$-Hom-Lie color algebras}

\author{Ibrahima Bakayoko, \\
\footnotesize{D\'epartement de Math\'ematiques, Universit\'e de N'Z\'er\'ekor\'e,} \\
\footnotesize{BP 50 N'Z\'er\'ekor\'e, Guin\'ee.} \\
\footnotesize{\text{ibrahimabakayoko27@gmail.com}, }\\
Sergei Silvestrov, \\
\footnotesize{Division of Applied Mathematics, School of Education, Culture and Communication,} \\
\footnotesize{M\"alardalen University, Box 883, 72123 V{\"a}ster{\aa}s, Sweden.} \\
\footnotesize{\text{sergei.silvestrov@mdh.se}}}

\date{October 15, 2019}

\maketitle

\abstract{The purpose of this paper is to generalize some results on $n$-Lie algebras and $n$-Hom-Lie algebras to $n$-Hom-Lie color algebras.
Then we introduce and give some constructions of $n$-Hom-Lie color algebras.
}

\vspace{0.5cm}
\noindent \textbf{Keywords}: n-Hom-Lie color algebras, color modules, averaging, semi-morphism, morphism. \\
{\bf Mathematics Subject Classification (2010):} 17A40,17A42,17B30,17B55

\section{Introduction}
\label{BakayokoSilv:sec:introduction}

The investigations of various $q$-deformations (quantum deformations) of Lie algebras began a period of rapid expansion in 1980's stimulated by introduction of quantum groups motivated by applications to the quantum Yang-Baxter equation, quantum inverse scattering methods and constructions of the quantum deformations of universal enveloping algebras of semi-simple Lie algebras. In
\cite{ChapBaSilnhomliecolor:AizawaSaito,ChapBaSilnhomliecolor:ChaiElinPop,ChapBaSilnhomliecolor:ChaiIsLukPopPresn,ChapBaSilnhomliecolor:ChaiKuLuk,ChapBaSilnhomliecolor:ChaiPopPres,ChapBaSilnhomliecolor:CurtrZachos1,ChapBaSilnhomliecolor:DamKu,ChapBaSilnhomliecolor:DaskaloyannisGendefVir,ChapBaSilnhomliecolor:Hu,ChapBaSilnhomliecolor:Kassel92,ChapBaSilnhomliecolor:LiuKQuantumCentExt,ChapBaSilnhomliecolor:LiuKQCharQuantWittAlg,ChapBaSilnhomliecolor:LiuKQPhDthesis}
various versions of $q$-deformed Lie algebras appeared in physical contexts such as string theory, vertex models in conformal field theory, quantum mechanics and quantum field theory in the context of $q$-deformations of infinite-dimensional algebras, primarily the $q$-deformed Heisenberg algebras \cite{ChapBaSilnhomliecolor:HelSilbookqHeis}, $q$-deformed oscillator algebras and $q$-deformed Witt and $q$-deformed Virasoro algebras, and some interesting $q$-deformations of the Jacobi identity for Lie algebras in these $q$-deformed algebras were observed.

Hom-Lie algebras and more general quasi-hom-Lie algebras were introduced first by Larsson, Hartwig and Silvestrov \cite{ChapBaSilnhomliecolor:HLS}, where
the general quasi-deformations and discretizations of Lie algebras of vector fields using more general $\sigma$-derivations (twisted derivations) and a general method for construction of deformations of Witt and Virasoro type algebras based on twisted derivations have been developed,
initially motivated by the $q$-deformed Jacobi identities observed for the $q$-deformed algebras in physics, along with $q$-deformed versions of homological algebra and discrete modifications of differential calculi. The general abstract quasi-Lie algebras and the subclasses of quasi-Hom-Lie algebras and Hom-Lie algebras as well as their general color (graded) counterparts have been introduced in \cite{ChapBaSilnhomliecolor:HLS,ChapBaSilnhomliecolor:LS1,ChapBaSilnhomliecolor:LSGradedquasiLiealg,ChapBaSilnhomliecolor:Czech:witt,ChapBaSilnhomliecolor:LS2}.
Subsequently, various classes of hom-Lie admissible algebras have been considered in \cite{ChapBaSilnhomliecolor:ms:homstructure}. In particular, in \cite{ChapBaSilnhomliecolor:ms:homstructure}, the hom-associative algebras have been introduced and shown to be hom-Lie admissible, that is leading to hom-Lie algebras using commutator map as new product, and in this sense constituting a natural generalization of associative algebras, as Lie admissible algebras leading to Lie algebras via commutator map as new product.
In \cite{ChapBaSilnhomliecolor:ms:homstructure}, moreover several other interesting classes of hom-Lie admissible algebras generalising some classes of non-associative algebras, as well as examples of finite-dimensional hom-Lie algebras have been described. Since these pioneering works \cite{ChapBaSilnhomliecolor:HLS,ChapBaSilnhomliecolor:LS1,ChapBaSilnhomliecolor:LSGradedquasiLiealg,ChapBaSilnhomliecolor:LS2,ChapBaSilnhomliecolor:LS3,ChapBaSilnhomliecolor:ms:homstructure}, hom-algebra structures have developed in a popular broad area with increasing number of publications in various directions.
In hom-algebra structures, defining algebra identities are twisted by linear maps. This same idea was then generalized to various types of algebras in \cite{ChapBaSilnhomliecolor:ms:homstructure}. Hom-algebra structures of a given type include their classical counterparts and open more possibilities for deformations, extensions of cohomological structures and representations (see for example
\cite{ChapBaSilnhomliecolor:homdeformation,ChapBaSilnhomliecolor:Hombiliform,ChapBaSilnhomliecolor:LS1,ChapBaSilnhomliecolor:shenghomrep} and references therein).

Ternary algebras and more generally $n$-ary Lie algebras first appeared in Nambu's generalization of Hamiltonian mechanics, using a ternary generalization of Poisson algebras. The mathematical algebraic foundations of Nambu mechanics have been developed by Takhtajan and Daletskii in  \cite{ChapBaSilnhomliecolor:DalTakh,ChapBaSilnhomliecolor:Takhtajan:foundgenNambuMech,ChapBaSilnhomliecolor:Takhtajan:cohomology}. Filippov, in \cite{ChapBaSilnhomliecolor:Filippov:nLie} introduced $n$-Lie algebras and then Kasymov \cite{ChapBaSilnhomliecolor:Kasymov:nLie} deeper investigated their properties.
This approach uses the interpretation of Jacobi identity expressing the fact that  the adjoint map is a derivation. There is also another type of $n$-ary Lie algebras, in which the $n$-ary Jacobi identity is the sum over $S_{2n-1}$ instead of $S_3$ in the binary case. One reason for studying such algebras was that $n$-Lie algebras introduced by Filippov were mostly rigid, and these algebras offered more possibilities to this regard. Besides Nambu mechanics, $n$-Lie algebras revealed to have many applications in physics.

The $n$-Lie algebras found their applications in many fields of mathematics and Physics. For instance, Takhtajan has developed the foundations of
the theory of Nambu-Poisson manifolds \cite{ChapBaSilnhomliecolor:Takhtajan:foundgenNambuMech}.
The general cohomology theory for $n$-Lie algebras and Leibniz n-algebras was established in \cite{ChapBaSilnhomliecolor:RM}.
The structure and classification theory of finite dimensional $n$-Lie algebras was given by Ling \cite{ChapBaSilnhomliecolor:LW}  and many
other authors. For more details of the theory and applications of $n$-Lie algebras, see \cite{ChapBaSilnhomliecolor:aip:review}
and references therein.
Generalized derivations of Lie color algebras and $n$-ary (color) algebras have been studied in
\cite{ChapBaSilnhomliecolor:CM, ChapBaSilnhomliecolor:KI1, ChapBaSilnhomliecolor:KI2, ChapBaSilnhomliecolor:KI3, ChapBaSilnhomliecolor:KP}, and
generalizations of derivations in connection with extensions and enveloping algebras of hom-Lie color algebras have been considered in \cite{ChapBaSilnhomliecolor:envelopalgcolhomLiealg,ChapBaSilnhomliecolor:exthomLiecoloralg}.
Derivations, L-modules, L-comodules and Hom-Lie quasi-bialgebras have been considered in \cite{ChapBaSilnhomliecolor:IBLmodcomodhomLiequasibialg,ChapBaSilnhomliecolor:IBLaplacehomLiequasibialg}.
In \cite{ChapBaSilnhomliecolor:JM}, Leibnitz $n$-algebras have been studied.

Hom-type generalization of $n$-ary algebras, such as $n$-Hom-Lie algebras and other $n$-ary Hom algebras of Lie type and associative type, were introduced in \cite{ChapBaSilnhomliecolor:AtMaSi:GenNambuAlg}, by twisting the identities defining them using a set of linear maps, together with the particular case where all these maps are equal and are algebra morphisms.
A way to generate examples of such algebras from non Hom-algebras of the same type is introduced.
Further properties, construction methods, examples, cohomology and central extensions of $n$-ary Hom-algebras have been considered in
\cite{ChapBaSilnhomliecolor:akms:ternary,ChapBaSilnhomliecolor:ams:ternary,ChapBaSilnhomliecolor:ams:n,ChapBaSilnhomliecolor:km:nary,ChapBaSilnhomliecolor:kms:nhominduced,ChapBaSilnhomliecolor:YauGenCom,ChapBaSilnhomliecolor:YauHomNambuLie}.
The construction of $(n+1)$-Lie algebras induced by $n$-Lie algebras using combination of bracket multiplication with a trace, motivated by the work of Awata et al. \cite{ChapBaSilnhomliecolor:almy:quantnambu} on the quantization of the Nambu brackets, was generalized using the brackets of general Hom-Lie algebra or $n$-Hom-Lie and trace-like linear forms depending on the linear maps defining the Hom-Lie or $n$-Hom-Lie algebras \cite{ChapBaSilnhomliecolor:ams:ternary,ChapBaSilnhomliecolor:ams:n}.

Properties and classification of $n$-ary algebras, including solvability and nilpotency, were studied in
\cite{ChapBaSilnhomliecolor:Bai:rlz3,ChapBaSilnhomliecolor:Bai:n,ChapBaSilnhomliecolor:Bai:nLie:clas,ChapBaSilnhomliecolor:Bai:nLie:claschar2,ChapBaSilnhomliecolor:RL,ChapBaSilnhomliecolor:RM1,ChapBaSilnhomliecolor:RM2,ChapBaSilnhomliecolor:RB,ChapBaSilnhomliecolor:Kasymov:nLie}.
Kasymov \cite{ChapBaSilnhomliecolor:Kasymov:nLie} pointed out that $n$-ary multiplication allows for several different definitions of solvability and nilpotency in $n$-Lie algebras, and studied their properties. The aim of this paper is to extend these definitions and properties to $n$-Hom-Lie algebras and to apply them to the case of $(n+1)$-Hom-Lie algebras induced by $n$-Hom-Lie algebras.

The purpose of this paper is to generalize some results on either $n$-Lie algebras or $n$-Hom-Lie algebras  to the case of $n$-Hom-Lie color algebras.
Then we introduce and give some constructions $n$-Hom-Lie color algebras.
Section \ref{BakayokoSilvestrov:sec:preliminaries} contains some necessary important basic notions and notations on graded spaces and algebras and $n$-ary algebras and used in other sections. Section \ref{BakayokoSilvestrov:sec:constrnhomcoloralg} presents some useful methods for construction of $n$-Hom-Lie color algebras.
In Section \ref{BakayokoSilvestrov:sec:modulesnhomcoloralg}, Hom-modules over $n$-Hom-Lie color algebras are considered.
Section \ref{BakayokoSilvestrov:sec:genderivationsnhomcoloralg} is devoted to generalized derivation of color Hom-algebras and their color Hom-subalgebras.

Throughout this paper, all graded linear spaces are assumed to be over a field $\mathbb{K}$ of characteristic different from 2.

\section{Preliminaries}
\label{BakayokoSilvestrov:sec:preliminaries}
This section contains necessary important basic notions and notations on graded spaces and algebras and $n$-ary algebras used in other sections.

\begin{definition} Let $G$ be an abelian group.
 \begin{enumerate}   \index{Linear space!graded}
  \item [1)] A linear space $V$ is said to be a $G$-graded if, there exists a family $(V_a)_{a\in G}$ of linear
subspaces of $V$ such that
$$V=\bigoplus_{a\in G} V_a.$$
\item [2)]
\index{Homogeneous element}
An element $x\in V$ is said to be homogeneous of degree $a\in G$ if $x\in V_a$. We denote $\mathcal{H}(V)$ the set of all homogeneous elements
in $V$.
\item [3)] Let $V=\oplus_{a\in G} V_a$ and $V'=\oplus_{a\in G} V'_a$ be two $G$-graded linear spaces. A linear mapping $f : V\rightarrow V'$ is said
to be homogeneous of degree $b$ if
$$f(V_a)\subseteq  V'_{a+b}, \quad \text{ for all }\quad a\in G.$$
If, $f$ is homogeneous of degree zero i.e. $f(V_a)\subseteq V'_{a}$ holds for any $a\in G$, then $f$ is said to be even.
 \end{enumerate}
\end{definition}
\begin{definition}
  \begin{enumerate}   \index{Algebra!graded}
\item[1)] An algebra $(A, \cdot)$ is said to be $G$-graded if its underlying linear space is $G$-graded i.e. $A=\bigoplus_{a\in G}A_a$, and if furthermore
$$A_a\cdot A_b\subseteq A_{a+b}, \quad \text{ for all } \quad a, b\in G.$$
\item[2)] \index{Morphism!graded algebras} A morphism $f : A\rightarrow A'$
of $G$-graded algebras $A$ and $A'$
is by definition an algebra morphism from $A$ to $A'$ which is, in addition an even mapping.
  \end{enumerate}
\end{definition}

\begin{definition}
 Let $G$ be an abelian group. A map $\varepsilon :G\times G\rightarrow {\bf \mathbb{K}^*}$ is called a skew-symmetric bicharacter \index{Bicharacter!skew-symmetric}
 on $G$ if the following identities hold for all $a, b, c\in G$:
\begin{enumerate}
\item [(i)] $\varepsilon(a, b)\varepsilon(b, a)=1$,
\item [(ii)] $\varepsilon(a, b+c)=\varepsilon(a, b)\varepsilon(a, c)$,
\item [(iii)]$\varepsilon(a+b, c)=\varepsilon(a, c)\varepsilon(b, c)$,
\end{enumerate}
\end{definition}

If $x$ and $y$ are two homogeneous elements of degree $a$ and $b$ respectively and $\varepsilon$ is a skew-symmetric bicharacter,
then we shorten the notation by writing $\varepsilon(x, y)$ instead of $\varepsilon(a, b)$.

\begin{example} Some standard examples of skew-symmetric bicharacters are:
 \begin{enumerate}
\item [1)] $G=\mathbb{Z}_2,\quad \varepsilon(i, j)=(-1)^{ij}$, or more generally
\begin{gather*} G=\mathbb{Z}_2^n=\{(\alpha_1, \dots, \alpha_n)| \alpha_i\in\mathbb{Z}_2 \}, \\
\varepsilon((\alpha_1, \dots, \alpha_n), (\beta_1, \dots, \beta_n)):= (-1)^{\alpha_1\beta_1+\dots+\alpha_n\beta_n}.
\end{gather*}
\item [2)] $G=\mathbb{Z}_2\times\mathbb{Z}_2,\quad \varepsilon((i_1, i_2), (j_1, j_2)=(-1)^{i_1j_2-i_2j_1}$,
\item [3)] $G=\mathbb{Z}\times\mathbb{Z} ,\quad \varepsilon((i_1, i_2), (j_1, j_2))=(-1)^{(i_1+i_2)(j_1+j_2)}$,
\item [4)] $G=\{-1, +1\} , \quad\varepsilon(i, j)=(-1)^{(i-1)(j-1)/{4}}$.
\end{enumerate}
\end{example}

\begin{definition}
An $n$-Lie algebra \index{Algebra!n-Lie} is a linear spaces $V$ equipped with $n$-ary operation which is skew-symmetric
for any pair of variables and satisfies the following identity:
\begin{equation}
\begin{array}{l}
 [x_1, \dots, x_{n-1}, [y_1, \dots, y_n]] =\\
= \displaystyle \sum_{i=1}^n[y_1, \dots, y_{i-1}, [x_1, \dots, x_{n-1}, y_i], y_{i+1}, \dots y_{n}].
\end{array} \label{BakayokoSilvestrov:F}
\end{equation}
\end{definition}

\begin{definition}
 An $n$-Hom-Lie color algebra \index{Algebra!n-Hom-Lie color} is a graded linear space $L=\oplus L_a, a\in G$ with an $n$-linear map
 $[\cdot \dots, \cdot]: L\times\dots \times L\rightarrow L$, a bicharacter $\varepsilon : G\times G\rightarrow \bf K^*$ and an even linear map
$\alpha : L\rightarrow L$ such that
\begin{eqnarray}
[x_1, \dots, x_i, x_{i+1}, \dots, x_n]&=&-\varepsilon(x_i, x_{i+1})[x_1, \dots, x_{i+1}, x_i, \dots, x_n], \\
&& \quad \quad \quad \quad \quad \quad  i=1,2, \dots n-1. \nonumber
\end{eqnarray}
\begin{equation} %
\begin{array}{l}
[\alpha(x_1), \dots, \alpha(x_{n-1}), [y_1, y_2, \dots, y_{n}]] = \\
=\displaystyle \sum_{i=1}^n\varepsilon(X, Y_i)[\alpha(y_1), \dots, \alpha(y_{i-1}), [x_1, \dots, x_{n-1}, y_i], \alpha(y_{i+1}),\dots, \alpha(y_{n})]
\end{array}
\end{equation}
where $x_i, y_j\in\mathcal{H}(L)$, $X=\sum_{i=1}^{n-1}x_i$,  $Y_i=\sum_{j=1}^iy_{j-1}$ and $y_0=e$.
\end{definition}
\begin{remark}
 Whenever $n=2$ (resp. $n=3$) we recover Hom-Lie color algebras \index{Algebra!Hom-Lie color} \index{Algebra!ternary Hom-Lie color} (resp. ternairy Hom-Lie color algebras).
\end{remark}

\begin{remark}
\begin{enumerate}
\item [1)] When $\alpha=id$, we get $n$-Lie color algebra.
\item [2)] When $G=\{e\}$ and $\alpha=id$, we get $n$-Lie algebra.
\item [3)] When $G=\{e\}$ and $\alpha\neq id$, we get $n$-Hom-Lie algebra.
\end{enumerate}
\end{remark}

\begin{definition}
 A morphism   $f : (L, [\cdot, \dots, \cdot], \varepsilon, \alpha)\rightarrow (L', [\cdot, \dots, \cdot]', \varepsilon, \alpha')$ of an $n$-Hom-Lie color algebras
 is an even linear map $f : L\rightarrow L$ such that
$f\circ\alpha=\alpha\circ f$  and for any $x_i\in\mathcal{H}(L)$,
$$f([x_1, \dots, x_n])=[f(x_1), \dots, f(x_n)]'$$
\end{definition}

\begin{definition}
1) An $n$-Hom-Lie color algebra $(L, [\cdot, \dots, \cdot], \varepsilon, \alpha)$ is said to be multiplicative \index{Algebra!multiplicative n-Hom-Lie color} if $\alpha$ is an endomorphism, i.e. a linear map on $L$ which is  also a homomorphism with respect to multiplication
$[\cdot, \dots, \cdot]$).\\
2) An $n$-Hom-Lie color algebra $(L, [\cdot, \dots, \cdot], \varepsilon, \alpha)$ is said to be regular \index{Algebra!regular n-Hom-Lie color} if $\alpha$ is an automorphism.\\
3) An $n$-Hom-Lie color algebra $(L, [\cdot, \dots, \cdot], \varepsilon, \alpha)$ is said to be involutive \index{Algebra!involutive n-Hom-Lie color} if $\alpha^2=id$.
\end{definition}

\begin{example}
 Let $G=\mathbb{Z}_2,\quad \varepsilon(i, j)=(-1)^{ij}$,
$L=L_0\oplus L_1=<e_2, e_4>\oplus <e_1, e_3>$,
$$[e_1, e_2, e_3]=e_2,\quad [e_1, e_2, e_4]=e_1,\quad [e_1, e_3, e_4]=[e_2, e_3, e_4]=0,$$
and
$\alpha(e_1)=e_3, \quad\alpha(e_2)=e_4, \quad \alpha(e_3)=\alpha(e_4)=0.$
Then $(L, [\cdot, \cdot, \cdot], \varepsilon, \alpha)$ is a $3$-Hom-Lie color algebra.
\end{example}

 \begin{example}
  Let $L$ be a graded linear space $$L=L_{(0, 0)}\oplus L_{(0, 1)}\oplus L_{(1, 0)}\oplus L_{(1, 1)}$$ with
$L_{(0, 0)}=<e_1, e_2>,\quad  L_{(0, 1)}=<e_3>,\quad  L_{(1, 0)}=<e_4>,\quad L_{(1, 1)}=<e_5>.$
The $4$-ary even linear multiplication $[\cdot, \cdot, \cdot, \cdot] : L\times L\times L\times L\rightarrow L$ defined for basis $\{e_i \mid i=1, \dots ,5\}$
by
\begin{gather*}
[e_2, e_3, e_4, e_5]=e_1, [e_1, e_3, e_4, e_5]=e_2,  [e_1, e_2, e_4, e_5]=e_3, \\
[e_1, e_2, e_3, e_4]=0, [e_1, e_2, e_3, e_5]=0
\end{gather*}
makes $L$ into the five dimensional $4$-Lie color algebra.

Now define on $(L, [\cdot, \cdot, \cdot, \cdot], \varepsilon)$ an even endomorphism $\alpha :L\rightarrow L$ by
$$\alpha(e_1)=e_2,\quad \alpha(e_2)=e_1, \quad \alpha(e_i)=e_i, \quad i=3, 4, 5.$$
Then $L_\alpha=(L, [\cdot, \cdot, \cdot, \cdot]_\alpha, \varepsilon, \alpha)$ is a $4$-Hom-Lie color algebra.
 Observe that $\alpha$ is involutive (hence bijective).
 \end{example}

\begin{definition}
A graded subspace $H$ of an $n$-Hom-Lie color algebra $L$ is a color Hom-subalgebra \index{Hom-subalgebra!color} of $L$ if
\begin{enumerate}
 \item [i)] $\alpha(H)\subseteq H$,
\item [ii)] $[H, H, \dots, H]\subseteq H$.
\end{enumerate}
\end{definition}

\begin{definition}
Let $L_1, L_2, \dots, L_n$ be Hom-subalgebras of an $n$-Hom-Lie color algebra $L$. Denote by $[L_1, L_2, \dots, L_n]$ the Hom-subalgebra of $L$
generated by all elements $[x_1, x_2, \dots, x_n]$, where $x_i\in L_i, i=1, 2, \dots, n$.
\begin{enumerate}
 \item [i)] The sequence $L_1, L_2, \dots, L_n, \dots$ defined by
\begin{gather*}
L_0=L,\quad L_1=[L_0, L_0, \dots, L_0],\quad L_2=[L_1, L_1, \dots, L_1], \dots,\\
L_n=[L_{n-1}, L_{n-1}, \dots, L_{n-1}], \dots
\end{gather*}
is called the derived sequence. \index{Derived sequence}
\item [ii)] The sequence $L^1, L^2, \dots, L^n, \dots$ defined by
\begin{gather*} L^0=L,\quad L^1=[L^0, L, \dots, L],\quad L^2=[L^1, L, \dots, L], \dots, \\
L^n=[L^{n-1}, L, \dots, L], \dots \end{gather*}
is called the descending central sequence. \index{Descending central sequence}
\item [iii)] The graded subspace $Z(L)$ defined by
\begin{eqnarray}
 Z(L)=\{x\in L| [x, L, L, \dots, L]=0\} \label{BakayokoSilvestrov:z}
\end{eqnarray}
is called the center \index{Center!n-Hom-Lie color algebras} of $L$.
\end{enumerate}
\end{definition}
\begin{definition}
\index{Hom-ideal!n-Hom-Lie color algebra}
A Hom-ideal $I$ of an $n$-Hom-Lie color algebra $L$ is a graded subspace of $L$ such that $\alpha(I)\subseteq I$ and $[I, L, \dots, L]\subseteq I$.
\end{definition}
\begin{theorem}
 Let $(L, [\cdot, \dots, \cdot], \varepsilon, \alpha)$ be an $n$-Hom-Lie color algebra with surjective twisting map $\alpha : L\rightarrow L$. Then,
$I_n, I^n$ and $Z(L)$ are Hom-ideals of $L$.
\end{theorem}
\begin{proof}
We only prove, by induction, that $I_n$ is a Hom-ideal. For this, suppose, first, that $I_{n-1}$ is a Hom-subalgebra of $L$ and show that $I_n$
is a Hom-subalgebra of $L$.
For any $y\in \mathcal{H}(I_n)$, there exist $y_1, y_2, \dots, y_n\in \mathcal{H}(I_{n-1})$, such that
$$y=[y_1, y_2, \dots, y_n].$$ So, $\alpha(y)=\alpha([y_1, y_2, \dots, y_n])=[\alpha(y_1), \alpha(y_2), \dots, \alpha(y_n)]$, which
belong to $I_n$ because $I_{n-1}$ is a Hom-subalgebra. That is $\alpha(I_n)\subseteq I_n$.

For any $y_i\in \mathcal{H}(I_{n})$, there exist $y_i^1, y_i^2, \dots, y_i^n\in I_{n-1}, i=1, 2, \dots, n$ such that
$$[y_1, y_2, \dots, y_n]=[[y_1^1, y_1^2, \dots, y_1^n], [y_2^1, \dots, y_2^n], \dots, [y_n^1, \dots, y_n^n]].$$
$I_{n-1}$ being a Hom-subalgebra, by hypotheses, $[y_i^1, y_i^2, \dots, y_i^n]\in I_{n-1}$ for $1\leq i\leq n$, and so
$[y_1, y_2, \dots, y_n]\in I_n$. Thus $I_n$ is a Hom-subalgebra.

Now, suppose that $I_{n-1}$ is a Hom-ideal. Let $x'_1, \dots, x'_{n-1}\in L, y\in I_n$, then there exist
$x_1, \dots, x_{n-1}\in L$, $y_1, \dots, y_n\in I_{n-1}$ such that
\begin{eqnarray}
&& [x'_1, \dots, x'_{n-1}, y]
=[\alpha(x_1), \dots, \alpha(x_{n-1}), [y_1, \dots, y_n]]\nonumber\\
&&=\sum_{i=1}^n\varepsilon(X, Y_i)[\alpha(y_1), \dots, \alpha(y_{i-1}), [x_1, \dots, x_{n-1}, y_i], \alpha(y_{i+1}), \alpha(y_n)]\nonumber
\end{eqnarray}
As $[x_1, \dots, x_{n-1}, y_i]\in I_{n-1}$, then $[x'_1, \dots, x'_{n-1}, y]\in I_n$. So, $I_n$ is a Hom-ideal of $L$.
\end{proof}

\section{Constructions of $n$-Hom-Lie color algebras}
\label{BakayokoSilvestrov:sec:constrnhomcoloralg}
In this section we present some useful methods for construction of $n$-Hom-Lie color algebras.

\begin{proposition}
Let $(L, [\cdot, \dots, \cdot], \varepsilon, \alpha)$ be an $n$-Hom-Lie color algebra and $\xi\in L_e$ such that $\alpha(\xi)=\xi$.
Then $(L, \{\cdot, \dots, \cdot\}, \varepsilon, \alpha)$ is an $(n-1)$-Hom-Lie color algebra with
$$\{x_1, \dots, x_{n-1}\}=[\xi, x_1, \dots, \dots, x_{n-1}].$$
\end{proposition}
\begin{proof} With conditions in the statement,
  \begin{multline}
\{\alpha(x_1), \dots, \alpha(x_{n-2}), \{y_1, \dots, y_{n-1}\}\} 
=[\xi,\alpha( x_1), \dots, \alpha(x_{n-2}), [\xi, y_1, \dots, y_{n-1}]]\nonumber\\
=[\alpha(\xi),\alpha( x_1), \dots, \alpha(x_{n-2}), [\xi, y_1, \dots, y_{n-1}]]\nonumber\\
=[[\xi, x_1, \dots, x_{n-2}, \xi], \alpha(y_1), \dots, \alpha(y_{n-1})]\nonumber\\
+\sum_{i=i}^{n-1}\varepsilon(X, Y_i)[\xi, \alpha(y_1), \dots, \alpha(y_{i-1}), [\xi, x_1, \dots, x_{n-2}, y_i], \alpha(y_{i+1}), \dots, \alpha(y_{n-1})]\nonumber\\
=\sum_{i=i}^{n-1}\varepsilon(X, Y_i)\{\alpha(y_1), \dots, \alpha(y_{i-1}), \{x_1, \dots, x_{n-2}, y_i\}, \alpha(y_{i+1}), \dots, \alpha(y_{n-1})\}\nonumber.
  \end{multline}
which completes the proof.
\end{proof}

\begin{corollary}
Let $(L, [\cdot, \dots, \cdot], \varepsilon, \alpha)$ be an $n$-Hom-Lie color algebra and $\xi_i\in L_e$ such that $\alpha(\xi_i)=\xi_i, i=1,2, \dots, k$.
Then $L_k=(L, \{\cdot, \dots, \cdot\}_k, \varepsilon, \alpha)$ is an $(n-k)$-Hom-Lie color algebra with
$\{x_1, \dots, x_{n-k}\}_k=[\xi_1, \dots, \xi_k, x_1\dots, \dots, x_{n-k}].$
\end{corollary}

\begin{corollary}
Let $(L, [\cdot, \dots, \cdot], \varepsilon)$ be an $n$-Lie color algebra and $\xi\in L_e$.

Then $(L, \{\cdot, \dots, \cdot\}, \varepsilon)$ is an $(n-1)$-Lie color algebra with
$$\{x_1, \dots, x_{n-1}\}=[\xi, x_1, \dots, \dots, x_{n-1}].$$
\end{corollary}

\begin{theorem}
Let $(L, [\cdot, \dots, \cdot], \varepsilon, \alpha)$ be an $n$-Hom-Lie color algebra and $\beta$ be an even endomorphism of $L$. Then
$L_\beta=(L, \{\cdot, \dots, \cdot\}=\beta[\cdot, \dots, \cdot], \varepsilon, \beta\alpha)$
is an $n$-Hom-Lie color algebra.
Moreover suppose that $(L', [\cdot, \dots, \cdot]', \varepsilon, \alpha')$ is another $n$-Hom-Lie color algebra and $\beta'$ be an even endomorphism of $L'$.
If $f : (L, [\cdot, \dots, \cdot], \varepsilon, \alpha)\rightarrow (L', [\cdot, \dots, \cdot]', \varepsilon, \alpha')$
is a morphism such that $f\beta=\beta'f$, then $f : L_\beta\rightarrow L_{\beta'}$ is also a morphism.
\end{theorem}
\begin{proof}
 First part is proved as follows:
\begin{eqnarray}
 &&\{\beta\alpha(x_1), \dots, \beta\alpha (x_{n-1}), \{y_1, \dots, y_n\}\}=\beta([\beta\alpha (x_1), \dots, \beta\alpha (x_{n-1}), \beta[y_1, \dots, y_n]])\nonumber\\
&&=\beta^2\Big([\alpha(x_1), \dots, \alpha(x_{n-1}), [y_1, y_2, \dots, y_{n}]]\Big)\nonumber\\
&&=\beta^2\Big(\sum_{i=1}^n\varepsilon(X, Y_i)[\alpha(y_1), \dots, \alpha(y_{i-1}),
[x_1, \dots, x_{n-1}, y_i], \alpha(y_{i+1}),\dots, \alpha(y_{n})]\Big)\nonumber\\
&&=\sum_{i=1}^n\varepsilon(X, Y_i)\beta^2\Big([\alpha(y_1), \dots, \alpha(y_{i-1}),
[x_1, \dots, x_{n-1}, y_i], \alpha(y_{i+1}),\dots, \alpha(y_{n})]\Big)\nonumber\\
&&=\sum_{i=1}^n\varepsilon(X, Y_i)\beta[\beta\alpha(y_1), \dots, \beta\alpha(y_{i-1}),
\beta[x_1, \dots, x_{n-1}, y_i], \beta\alpha(y_{i+1}),\dots, \beta\alpha(y_{n})]\nonumber\\
&&=\sum_{i=1}^n\varepsilon(X, Y_i)\{\beta\alpha(y_1), \dots, \beta\alpha(y_{i-1}),
\{x_1, \dots, x_{n-1}, y_i\}, \beta\alpha(y_{i+1}),\dots, \beta\alpha(y_{n})\}\nonumber.
\end{eqnarray}
Second part is proved as follows:
\begin{eqnarray}
 f(\{x_1, \dots, x_n\})&=&f([x_1, \dots, x_n]_\beta)
=f\beta[x_1, \dots, x_n]=f[\beta(x_1), \dots, \beta(x_n)]\nonumber\\
&=&[f\beta(x_1), \dots, f\beta(x_n)]'=[\beta'f(x_1), \dots, \beta'f(x_n)]'\nonumber\\
&=&\beta'[f(x_1), \dots, f(x_n)]'=[f(x_1), \dots, f(x_n)]'_{\beta'}\nonumber\\
&=&\{ f(x_1), \dots, f(x_n)\}'\nonumber
\end{eqnarray}
This completes the proof.
\end{proof}
Taking $\beta=\alpha^n$ leads to the following statement.
\begin{corollary}
Let $(L, [\cdot, \dots, \cdot], \varepsilon, \alpha)$ be a multiplicative $n$-Hom-Lie color algebra. Then, for any positive integer $n$,
$(L, \alpha^n[\cdot, \dots, \cdot], \varepsilon, \alpha^{n+1})$ is also an $n$-Hom-Lie color algebra.
\end{corollary}
Taking $\beta=\alpha$ and $\alpha=id$ leads to the following statement.
\begin{corollary}
 Let $(L, [\cdot, \dots, \cdot], \varepsilon)$ be an $n$-Lie color algebra and $\alpha$ be an even endomorphism of $L$. Then
$L_\alpha=(L, \{\cdot, \dots, \cdot\}=\alpha[\cdot, \dots, \cdot], \varepsilon, \alpha)$ is a multiplicative $n$-Hom-Lie color algebra.
\end{corollary}
Taking $\beta\in Aut(L), \beta=\alpha^{-1}$ leads to the following statement.
\begin{corollary}
 Let $(L, [\cdot, \dots, \cdot], \varepsilon, \alpha)$ be a regular $n$-Hom-Lie color algebra. Then
$$L_{\alpha^{-1}}=(L, \{\cdot, \dots, \cdot\}=\alpha^{-1}[\cdot, \dots, \cdot], \varepsilon)$$
is an $n$-Lie color algebra.
\end{corollary}
Taking $\beta=\alpha$ leads to the following statement.
\begin{corollary}
  Let $(L, [\cdot, \dots, \cdot], \varepsilon, \alpha)$ be an involutive $n$-Hom-Lie color algebra. Then
$$L_\beta=(L, \{\cdot, \dots, \cdot\}=\alpha[\cdot, \dots, \cdot], \varepsilon)$$
is an $n$-Lie color algebra.
\end{corollary}

\begin{theorem}\label{BakayokoSilvestrov:tp}
Let $(A, \cdot)$ be a commutative associative algebra and $(L, [\cdot, \dots, \cdot], \varepsilon, \alpha)$ be an $n$-Hom-Lie color algebra. The tensor
product $A\otimes L=\sum_{g\in G}(A\otimes L)_g=\sum_{g\in G}A\otimes L_g$ with the bracket
$[a_1\otimes x_1, \dots, a_n\otimes x_n]'=a_1\dots a_n\otimes [x_1, \dots, x_n],$
the even linear map
$\alpha'(a\otimes x):=a\otimes \alpha(x)$
and the bicharacter
$\varepsilon(a+x, b+y)=\varepsilon(x, y), \forall a, b\in A, \forall x, y\in \mathcal{H}(L),$
is an $n$-Hom-Lie color algebra.
\end{theorem}
\begin{proof}
 It follows from a straightforward computation.
\end{proof}

In the next definition, we introduce an element of the centroid (or semi-morphism) for $n$-Hom-Lie color algebra.
\begin{definition}
 A semi-morphism \index{Semi-morphism!n-Hom-Lie color algebra}  of an $n$-Hom-Lie color algebra $(L, [\cdot, \dots, \cdot], \varepsilon, \alpha)$ is an even linear map $\beta : L\rightarrow L$ such
that $\beta\alpha=\alpha\beta$ and
$\beta[x_1, \dots, x_n]=[\beta(x_1), x_2, \dots, x_n].$
\end{definition}
\begin{remark}
Due to $\varepsilon$-skew-symmetry, $\beta[x_1, \dots, x_n]=[x_1, \dots\beta(x_i), \dots, x_n].$ \end{remark}
\begin{theorem}\label{BakayokoSilvestrov:tch}
 Let $(L, [\cdot, \dots, \cdot], \varepsilon, \alpha)$ be an $n$-Hom-Lie color algebra and $\beta : L\rightarrow L$ a semi-morphism of $L$.
Define a new multiplication $\{\cdot, \dots, \cdot\}: L\times\dots\times L\rightarrow L$ by
$$\{x_1, \dots, x_n\}=[x_1, \dots, \beta(x_i)\dots, x_n]$$
Then $(L, \{\cdot, \dots, \cdot\}, \varepsilon, \alpha)$ is also an $n$-Hom-Lie color algebra.
\end{theorem}
\begin{proof}
 The proof can be obtained as follows:
 \begin{eqnarray} &&\{\alpha(x_1), \dots, \alpha(x_{n-1}), \{y_1, \dots, y_n\}\}=\nonumber\\
&&=[\alpha(x_1), \dots, \beta\alpha(x_i), \dots, \alpha(x_{n-1}), [y_1, \dots, \beta(y_j), \dots, y_n]]\nonumber\\
&&=[\alpha(x_1), \dots, \alpha\beta(x_i), \dots, \alpha(x_{n-1}), [y_1, \dots, \beta(y_j), \dots, y_n]]\nonumber\\
&&
=\sum_{k<j}\varepsilon(X, Y_k)[\alpha(y_1), \dots, \alpha(y_{k-1}),[x_1, \dots, \beta(x_i), \dots, x_{n-1}, y_k], \nonumber \\
&& \hspace{6cm} \alpha(y_{k+1}), \dots, \alpha\beta(y_j), \dots \alpha(y_n)]
\nonumber\\
&&+\varepsilon(X, Y_j)[\alpha(y_1), \dots, \alpha(y_{j-1}),
[x_1, \dots, \beta(x_i), \dots, x_{n-1}, \beta(y_j)], \alpha(y_{j+1}), \dots, \alpha(y_n)]\nonumber\\
&&
+\sum_{k>j}\varepsilon(X, Y_k)[\alpha(y_1), \dots, \alpha\beta(y_{j}), \dots, \alpha(y_{k-1})
[x_1, \dots, \beta(x_i), \dots, x_{n-1}, y_k],  \nonumber \\
&& \hspace{8cm} \alpha(y_{k+1}), \dots, \alpha(y_n)]
\nonumber\\
&&
=\sum_{k<j}\varepsilon(X, Y_k)[\alpha(y_1), \dots, \alpha(y_{k-1}),
\{x_1, \dots, x_i, \dots, x_{n-1}, y_k\},  \nonumber \\
&& \hspace{6cm}  \alpha(y_{k+1}), \dots, \beta\alpha(y_j), \dots \alpha(y_n)]
\nonumber\\
&&+\varepsilon(X, Y_j)[\alpha(y_1), \dots, \alpha(y_{j-1}),
\beta(\{x_1, \dots, x_i, \dots, x_{n-1}, y_j\}), \alpha(y_{j+1}), \dots, \alpha(y_n)]\nonumber\\
&&+\sum_{k>j}\varepsilon(X, Y_k)[\alpha(y_1), \dots, \beta\alpha(y_{j}), \dots, \alpha(y_{k-1}),
\{x_1, \dots, x_i, \dots, x_{n-1}, y_k\}, \nonumber\\
&& \hspace{8cm} \alpha(y_{k+1}), \dots, \alpha(y_n)]\nonumber\\
&&=\sum_{k<j}\varepsilon(X, Y_k)\{\alpha(y_1), \dots, \alpha(y_{k-1}),
\{x_1, \dots, x_i, \dots, x_{n-1}, y_k\}, \nonumber \\
&& \hspace{6cm}  \alpha(y_{k+1}), \dots, \alpha(y_j), \dots \alpha(y_n)\}\nonumber\\
&&+\varepsilon(X, Y_j)\{\alpha(y_1), \dots, \alpha(y_{j-1}),
\{x_1, \dots, x_i, \dots, x_{n-1}, y_j\}, \alpha(y_{j+1}), \dots, \alpha(y_n)\}\nonumber\\
&&+\sum_{k>j}\varepsilon(X, Y_k)\{\alpha(y_1), \dots, \alpha(y_{j}), \dots, \alpha(y_{k-1})
\{x_1, \dots, x_i, \dots, x_{n-1}, y_k\}, \nonumber \\
&& \hspace{8cm} \alpha(y_{k+1}), \dots, \alpha(y_n)\}\nonumber.
 \end{eqnarray}
This completes the proof.
 \end{proof}
\begin{corollary}
 Let $(L, [\cdot, \dots, \cdot], \varepsilon)$ be an $n$-Lie color algebra and $\alpha : L\rightarrow L$ a semi-morphism of $L$.
Then $(L, \{\cdot, \dots, \cdot\}, \varepsilon)$ is  another $n$-Lie color algebra, with
$$\{x_1, \dots, x_n\}=[x_1, \dots, \alpha(x_i)\dots, x_n]$$
\end{corollary}
\begin{definition}
 Let $(L, [\cdot, \dots, \cdot], \varepsilon, \alpha)$ be an $n$-Hom-Lie color algebra. An averaging operator
 of an $n$-Hom-Lie color algebra $L$ \index{Averaging operator!n-Hom-Lie color algebra} is an even linear map
 $\beta : L\rightarrow L$ such that
\begin{enumerate}
 \item [(1)] $\beta\alpha=\alpha\beta$
\item [(2)] $\beta[x_1, \dots, \beta(x_i), \dots, x_n]=[x_1, \dots, \beta(x_i), \dots, \beta(x_j), \dots, x_n]$
\end{enumerate}

\end{definition}
\begin{theorem}
 Let $(L, [\cdot, \dots, \cdot], \varepsilon)$ be an $n$-Lie color algebra and $\alpha : L\rightarrow L$ an averaging operator of $L$.
Define a new multiplication $\{\cdot, \dots, \cdot\}: L\times\dots\times L\rightarrow L$ by
$$\{x_1, \dots, x_n\}=[x_1, \dots, \alpha(x_i)\dots, x_n].$$
Then $L_\alpha=(L, \{\cdot, \dots, \cdot\}, \varepsilon, \alpha)$ is an $n$-Hom-Lie color algebra.
\end{theorem}
\begin{proof} The proof can be obtained as follows:
 \begin{eqnarray}
  &&\hspace{-1,3cm} \{\alpha(x_1), \dots, \alpha(x_{n-1}), \{y_1, \dots, y_n\}\} = [\alpha(x_1), \dots, \alpha^2(x_i), \dots, \alpha(x_{n-1}), [y_1, \dots, \alpha(y_j), \dots, y_n]]\nonumber\\
&&=\sum_{k<j}\varepsilon(X, Y_k)[\alpha(y_1), \dots, \alpha(y_{k-1}),
[x_1, \dots, \alpha(x_i), \dots, x_{n-1}, y_k], \nonumber \\
&& \hspace{6cm}
\alpha(y_{k+1}), \dots, \alpha^2(y_j), \dots \alpha(y_n)]\nonumber\\
&&+\varepsilon(X, Y_j)[\alpha(y_1), \dots, \alpha(y_{j-1}),
[x_1, \dots, \alpha(x_i), \dots, x_{n-1}, \alpha(y_j)],  \nonumber \\
&& \hspace{6cm}
\alpha(y_{j+1}), \dots, \alpha(y_n)]\nonumber\\
&&+\sum_{k>j}\varepsilon(X, Y_k)[\alpha(y_1), \dots, \alpha^2(y_{j}), \dots, \alpha(y_{k-1})
[x_1, \dots, \alpha(x_i), \dots, x_{n-1}, y_k],
\nonumber \\
&& \hspace{8cm}
\alpha(y_{k+1}), \dots, \alpha(y_n)]\nonumber\\
&&=\sum_{k<j}\varepsilon(X, Y_k)\{\alpha(y_1), \dots, \alpha(y_{k-1}),
\{x_1, \dots, x_i, \dots, x_{n-1}, y_k\},
\nonumber \\
&& \hspace{6cm}
\alpha(y_{k+1}), \dots, \alpha(y_j), \dots \alpha(y_n)\}\nonumber\\
&&+\varepsilon(X, Y_j)[\alpha(y_1), \dots, \alpha(y_{j-1}),
\alpha(\{x_1, \dots, x_i, \dots, x_{n-1}, y_j\}), \alpha(y_{j+1}), \dots, \alpha(y_n)]\nonumber\\
&&\quad+\sum_{k>j}\varepsilon(X, Y_k)[\alpha(y_1), \dots, \alpha(y_{j}), \dots, \alpha(y_{k-1}),
\{x_1, \dots, x_i, \dots, x_{n-1}, y_k\},
\nonumber \\
&& \hspace{8cm}
\alpha(y_{k+1}), \dots, \alpha(y_n)]\nonumber\\
&&=\sum_{k<j}\varepsilon(X, Y_k)\{\alpha(y_1), \dots, \alpha(y_{k-1}),
\{x_1, \dots, x_i, \dots, x_{n-1}, y_k\},
\nonumber \\
&& \hspace{6cm}
\alpha(y_{k+1}), \dots, \alpha(y_j), \dots \alpha(y_n)\}\nonumber\\
&&+\varepsilon(X, Y_j)\{\alpha(y_1), \dots, \alpha(y_{j-1}),
\{x_1, \dots, x_i, \dots, x_{n-1}, y_j\}, \alpha(y_{j+1}), \dots, \alpha(y_n)\}\nonumber\\
&&+\sum_{k>j}\varepsilon(X, Y_k)\{\alpha(y_1), \dots, \alpha(y_{j}), \dots, \alpha(y_{k-1})
\{x_1, \dots, x_i, \dots, x_{n-1}, y_k\},
\nonumber \\
&& \hspace{8cm}
\alpha(y_{k+1}), \dots, \alpha(y_n)\}\nonumber.
 \end{eqnarray}
This ends the proof.
 \end{proof}

\begin{theorem}
 Let $(L, [\cdot, \dots, \cdot], \varepsilon, \alpha)$ be an $n$-Hom-Lie color algebra and $\beta : L\rightarrow L$ an averaging operator of $L$.
Define a new multiplication $\{\cdot, \dots, \cdot\}: L\times\dots\times L\rightarrow L$ by
$$\{x_1, \dots, x_n\}=[x_1, \dots, \beta(x_i)\dots, x_n]$$
Then $(L, \{\cdot, \dots, \cdot\}, \varepsilon, \alpha)$ is also an $n$-Hom-Lie color algebra.
\end{theorem}
\begin{proof} It is similar to the one of Theorem \ref{BakayokoSilvestrov:tch}.
 \end{proof}
Taking $\alpha=id$, yields the following statement.
\begin{corollary}
 Let $(L, [\cdot, \dots, \cdot], \varepsilon)$ be an $n$-Lie color algebra and $\alpha : L\rightarrow L$ an averaging operator of $L$.
Then $(L, \{\cdot, \dots, \cdot\}, \varepsilon)$ is  another $n$-Lie color algebra, with
$$\{x_1, \dots, x_n\}=[x_1, \dots, \alpha(x_i)\dots, x_n]$$
\end{corollary}
Taking $\alpha=\beta$, yields the following statement.
\begin{corollary}
 Let $(L, [\cdot, \dots, \cdot], \varepsilon, \alpha)$ be an $n$-Hom-Lie color algebra and $\alpha : L\rightarrow L$ an averaging operator.
Define a new multiplication $\{\cdot, \dots, \cdot\}: L\times\dots\times L\rightarrow L$ by
$$\{x_1, \dots, x_n\}=[x_1, \dots, \alpha(x_i)\dots, x_n]$$
Then $(L, \{\cdot, \dots, \cdot\}, \varepsilon, \alpha)$ is also an $n$-Hom-Lie color algebra.
\end{corollary}

\begin{theorem}
 Let $(L, [\cdot, \dots, \cdot], \varepsilon, \alpha)$ be an $n$-Hom-Lie color algebra and $\beta : L\rightarrow L$ an averaging operator of $L$.
Then the new bracket $\{\cdot, \dots, \cdot\}: L\times\dots\times L\rightarrow L$ makes $L$ into an $n$-Hom-Lie color algebra with
$$\{x_1, \dots, x_n\}=[x_1, \dots, \beta(x_i)\dots, \beta(x_j), \dots, x_n].$$
\end{theorem}
\begin{proof} The proof is obtained as follows:
 \begin{eqnarray}
&&\{\alpha(x_1), \dots, \alpha(x_{n-1}), \{y_1, \dots, y_n\}\}=\nonumber\\
&&=[\alpha(x_1), \dots, \beta\alpha(x_i), \dots, \beta\alpha(x_j), \dots, \alpha(x_{n-1}),
[y_1, \dots, \beta(y_k), \dots, \beta(y_l), \dots, y_n]]\nonumber\\
&&=\sum_{m<k}\varepsilon(X, Y_m)[\alpha(y_1), \dots, \alpha(y_{m-1}),
[x_1, \dots, \beta(x_i), \dots, \beta(x_j), \dots, x_{n-1}, y_m], \nonumber \\
&& \hspace{4.5cm}
\alpha(y_{m+1}), \dots, \alpha\beta(y_k), \dots, \alpha\beta(y_l), \dots, \alpha(y_n)]\nonumber\\
&& \qquad+\varepsilon(X, Y_k)[\alpha(y_1), \dots, \alpha(y_{k-1}),
[x_1, \dots, \beta(x_i), \dots, \beta(x_j), \dots, x_{n-1}, \beta(y_k)],\nonumber\\
&&\hspace{5cm} \alpha(y_{k+1}), \dots, \alpha\beta(y_l), \dots, \alpha(y_n)]\nonumber\\
&& \qquad +\sum_{k<m<l}\varepsilon(X, Y_m)[\alpha(y_1), \dots, \alpha\beta(y_{k}), \dots, \alpha(y_{l-1}),
\nonumber \\
&& \hspace{1cm}
[x_1, \dots, \beta(x_i), \dots, \beta(x_j), \dots, x_{n-1}, y_m], \alpha(y_{m+1}), \dots, \alpha\beta(y_{l}), \dots, \alpha(y_n)]\nonumber\\
&&\qquad +\varepsilon(X, Y_l)[\alpha(y_1), \dots, \alpha\beta(y_{k}),\dots, \alpha(y_{l-1}),
\nonumber\\
&&\hspace{2.5cm}
[x_1, \dots, \beta(x_i), \dots, \beta(x_j), \dots, x_{n-1}, \beta(y_l)],\alpha(y_{l+1}), \dots, \alpha(y_n)]\nonumber\\
&&\qquad +\sum_{m>l}\varepsilon(X, Y_m)[\alpha(y_1), \dots, \alpha\beta(y_{k}), \dots, \alpha\beta(y_{l}), \dots, \alpha(y_{m-1}),
\nonumber\\
&&\hspace{2.5cm}
[x_1, \dots, \beta(x_i), \dots, \beta(x_j), \dots, x_{n-1}, y_m],\alpha(y_{m+1}), \dots,  \alpha(y_n)]\nonumber\\
&&=\sum_{m<k}\varepsilon(X, Y_m)\{\alpha(y_1), \dots, \alpha(y_{m-1}),
\{x_1, \dots, x_i, \dots, x_j, \dots, x_{n-1}, y_m\},
\nonumber\\
&& \hspace{5cm}
\alpha(y_{m+1}),\dots, \alpha(y_k), \dots, \alpha(y_l), \dots, \alpha(y_n)\}\nonumber\\
&&\qquad +\varepsilon(X, Y_k)[\alpha(y_1), \dots, \alpha(y_{k-1}),
\beta([x_1, \dots, \beta(x_i), \dots, \beta(x_j), \dots, x_{n-1}, y_k]),\nonumber\\
&&\hspace{6cm} \alpha(y_{k+1}), \dots, \beta\alpha(y_l), \dots, \alpha(y_n)]\nonumber\\
&&\qquad +\sum_{k<m<l}\varepsilon(X, Y_m)\{\alpha(y_1), \dots, \alpha(y_{k}), \dots, \alpha(y_{m-1}),
\nonumber\\
&&\hspace{2cm}
\{x_1, \dots, x_i, \dots, x_j, \dots, x_{n-1}, y_m\},\alpha(y_{m+1}), \dots, \alpha(y_{l}), \dots, \alpha(y_n)]\nonumber\\
&&\qquad +\varepsilon(X, Y_l)[\alpha(y_1), \dots, \beta\alpha(y_{k}),\dots, \alpha(y_{l-1}),
\nonumber\\
&&\hspace{2.5cm}
\beta([x_1, \dots, \beta(x_i), \dots, \beta(x_j), \dots, x_{n-1}, y_l]), \alpha(y_{l+1}), \dots, \alpha(y_n)]\nonumber\\
&& \qquad +\sum_{m>l}\varepsilon(X, Y_m)\{\alpha(y_1), \dots, \alpha(y_{k}), \dots, \alpha(y_{l}), \dots, \alpha(y_{m-1}),
\nonumber\\
&&\hspace{3cm}
\{x_1, \dots, x_i, \dots, x_j, \dots, x_{n-1}, y_m\},\alpha(y_{m+1}), \dots,  \alpha(y_n)\}\nonumber
\end{eqnarray}
\begin{eqnarray}
&&=\sum_{m<k}\varepsilon(X, Y_m)\{\alpha(y_1), \dots, \alpha(y_{m-1}),
\{x_1, \dots, x_i, \dots, x_j, \dots, x_{n-1}, y_m\},
\nonumber\\
&&
\hspace{5cm}
\alpha(y_{m+1}), \dots,\alpha(y_k), \dots, \alpha(y_l), \dots, \alpha(y_n)\}\nonumber\\
&& \qquad +\varepsilon(X, Y_k)\{\alpha(y_1), \dots, \alpha(y_{k-1}),
\{x_1, \dots, x_i, \dots, x_j, \dots, x_{n-1}, y_k\},\nonumber\\
&&\hspace{6cm} \alpha(y_{k+1}), \dots, \alpha(y_l), \dots, \alpha(y_n)\}\nonumber\\
&& \qquad +\sum_{k<m<l}\varepsilon(X, Y_m)\{\alpha(y_1), \dots, \alpha(y_{k}), \dots, \alpha(y_{m-1}),
\nonumber\\
&&\hspace{2cm}
\{x_1, \dots, x_i, \dots, x_j, \dots, x_{n-1}, y_m\},\alpha(y_{m+1}), \dots, \alpha(y_{l}), \dots, \alpha(y_n)\}\nonumber\\
&&+\varepsilon(X, Y_l)\{\alpha(y_1), \dots, \alpha(y_{k}),\dots, \alpha(y_{l-1}),
\{x_1, \dots, x_i, \dots, x_j, \dots, x_{n-1}, y_l\}),\nonumber\\
&&\hspace{5cm} \alpha(y_{l+1}), \dots, \alpha(y_n)\}\nonumber\\
&&+\sum_{m>l}\varepsilon(X, Y_m)\{\alpha(y_1), \dots, \alpha(y_{k}), \dots, \alpha(y_{l}), \dots, \alpha(y_{m-1}),
\nonumber\\
&&\hspace{2cm}
\{x_1, \dots, x_i, \dots, x_j, \dots, x_{n-1}, y_m\},\alpha(y_{m+1}), \dots,  \alpha(y_n)\}\nonumber.
 \end{eqnarray}
This finishes the proof.
 \end{proof}


\section{Hom-modules over $n$-Hom-Lie color algebras}
\label{BakayokoSilvestrov:sec:modulesnhomcoloralg}
In this section we consider Hom-modules over $n$-Hom-Lie color algebras.

\begin{definition}
Let $G$ be an abelian group. A Hom-module \index{Hom-module} is a pair $(M,\alpha_M)$ in which $M$ is a $G$-graded linear space and $\alpha_M: M\longrightarrow M$ is an even linear map.
\end{definition}

\begin{definition}
  Let $(L, [\cdot, \dots, \cdot], \varepsilon, \alpha)$ be an $n$-Hom-Lie color algebra and $(M, \alpha_M)$ a Hom-module. The Hom-module $(M, \alpha_M)$
 is called an  $n$-Hom-Lie module \index{Hom-module!n-Hom-Lie} over $L$ if there are  $n$ polylinear maps:
$$\omega_i : L\otimes \dots L\otimes\underbrace{M}_{i}\otimes L\otimes\dots\otimes L\rightarrow M, \quad i=1, 2, \dots, n$$
such that, for any $x_i, y_i\in\mathcal{H}(L)$ and $m\in\mathcal{H}(M)$,
\begin{enumerate}
 \item[a)]$\omega_i(x_1, \dots, x_{i-1}, m, x_{i+1}, \dots, x_n)$ is a $\varepsilon$-skew-symmetric by all $x$-type arguments.
 \item[b)]$\omega_i(x_1, \dots, x_{i-1}, m, x_{i+1}, \dots, x_n)=-\varepsilon(m, x_{i+1})\omega_i(x_1, \dots, x_{i-1}, x_{i+1}, m\dots, x_n)$ \\
for $i=1, 2, \dots, n-1$.
 \item[c)]
$\omega_n(\alpha(x_1), \dots, \alpha(x_{n-1}), \omega_n(y_1, \dots, y_{n-1}, m))=$
\begin{eqnarray}
&&=\sum_{i=1}^{n-1}\varepsilon(X, Y_i)\omega_n(\alpha(y_1), \dots, \alpha(y_{i-1}),
[x_1, \dots, x_{n-1}, y_i], \alpha(y_{i+1}), \dots, \alpha_M(m))\nonumber\\
&&\qquad+\varepsilon(X, Y_n)\omega_n(\alpha(y_1), \dots, \alpha(y_{n-1}), \omega_n(x_1, \dots, x_{n-1}, m)),\nonumber
\end{eqnarray}
where $x_i, y_j\in\mathcal{H}(L)$, $X=\sum_{i=1}^{n-1}x_i$,  $Y_i=\sum_{j=1}^iy_{j-1}, y_0=e$ and $m\in\mathcal{H}(M)$.
\item[d)]
$\omega_{n-1}(\alpha(x_1), \dots, \alpha(x_{n-2}), \alpha_M(m), [y_1, \dots, y_{n}])=$
\begin{eqnarray}
&&=\sum_{i=1}^{n}\varepsilon(X, Y_i)\omega_i(\alpha(y_1), \dots, \alpha(y_{i-1}),
\omega_{n-1}(x_1, \dots, x_{n-2}, m, y_i),
\nonumber \\
&&
\hspace{5cm}
\alpha(y_{i+1}), \dots, \alpha(y_n)]\nonumber,
\end{eqnarray}
where $x_i, y_j\in\mathcal{H}(L)$, $X=\sum_{i=1}^{n-2}x_i+m$,  $Y_i=\sum_{j=1}^iy_{j-1}, y_0=e$ and $m\in\mathcal{H}(M)$.
\end{enumerate}
\end{definition}

\begin{example}
 Any $n$-Hom-Lie color algebra $(L, [\cdot, \dots, \cdot], \varepsilon, \alpha)$ is an $n$-Hom-Lie module over itself by taking $M=L$, $\alpha_M=\alpha$
and $\omega_i(\cdot,\dots, \cdot)=[\cdot, \dots, \cdot]$.
\end{example}

\begin{theorem}
  Let $(L, [\cdot, \dots, \cdot], \varepsilon, \alpha)$ be an $n$-Hom-Lie color algebra, and let $(M, \alpha_M, \omega_i)$ be an $n$-Hom-Lie color module
and $\beta : L\rightarrow L$ be an endomorphism. Define
$$\tilde\omega_i=(\beta, \dots, \beta, \underbrace{id}_i, \beta, \dots, \beta), i=1, 2, \dots, n.$$
Then $(M, \alpha_M, \tilde\omega_i)$ is an $n$-Hom-Lie color module.
\end{theorem}
\begin{proof}
 The items $a)$ and $b)$ are obvious. So we only prove $c)$, item $d)$ being proved similarly.
\begin{eqnarray}
&&\tilde\omega_n(\alpha(x_1), \dots, \alpha(x_{n-1}), \tilde\omega_n(y_1, \dots, y_{n-1}, m))=\nonumber\\
&&=\omega_n(\beta\alpha(x_1), \dots, \beta\alpha(x_{n-1}), \omega_n(\beta(y_1), \dots, \beta(y_{n-1}), m))\nonumber\\
&&=\omega_n(\alpha\beta(x_1), \dots, \alpha\beta(x_{n-1}), \omega_n(\beta(y_1), \dots, \beta(y_{n-1}), m))\nonumber\\
&&=\sum_{i=1}^{n-1}\varepsilon(X, Y_i)\omega_n(\alpha\beta(y_1), \dots, \alpha\beta(y_{i-1}),
[\beta(x_1), \dots, \beta(x_{n-1}), \beta(y_i)], \nonumber\\
&& \hspace{7cm} \alpha\beta(y_{i+1}), \dots, \alpha_M(m)) \nonumber\\
&& \qquad +\varepsilon(X, Y_n)\omega_n(\alpha\beta(y_1), \dots, \alpha\beta(y_{n-1}), \omega_n(\beta(x_1), \dots, \beta(x_{n-1}), m))\nonumber\\
&&=\sum_{i=1}^{n-1}\varepsilon(X, Y_i)\omega_n(\beta\alpha(y_1), \dots, \beta\alpha(y_{i-1}),
\beta([x_1, \dots, x_{n-1}, y_i]),
\nonumber\\
&& \hspace{7cm}
\beta\alpha(y_{i+1}), \dots, \alpha_M(m))\nonumber\\
&&\qquad+\varepsilon(X, Y_n)\omega_n(\beta\alpha(y_1), \dots, \beta\alpha(y_{n-1}), \omega_n(\beta(x_1), \dots, \beta(x_{n-1}), m)),\nonumber\\
&&=\sum_{i=1}^{n-1}\varepsilon(X, Y_i)\omega_n(\beta\otimes\dots\otimes\beta\otimes id)(\alpha(y_1), \dots, \alpha(y_{i-1}),
[x_1, \dots, x_{n-1}, y_i],
\nonumber\\
&& \hspace{7.5cm}
\alpha(y_{i+1}), \dots, \alpha_M(m))\nonumber\\
&& \qquad +\varepsilon(X, Y_n)\omega_n(\beta\otimes\dots\otimes\beta\otimes id)\Big(id\otimes \dots\otimes id\otimes
 \omega_n(\beta\otimes\dots\otimes\beta\otimes id)\Big)
 \nonumber\\
&& \hspace{5cm}
(\alpha(y_1), \dots, \alpha(y_{n-1}), x_1, \dots, x_{n-1}, m)\nonumber\\
&&=\sum_{i=1}^{n-1}\varepsilon(X, Y_i)\tilde\omega_n(\alpha(y_1), \dots, \alpha(y_{i-1}),
[x_1, \dots, x_{n-1}, y_i], \alpha(y_{i+1}), \dots, \alpha_M(m))\nonumber\\
&&\qquad+\varepsilon(X, Y_n)\tilde\omega_n(\alpha(y_1), \dots, \alpha(y_{n-1}),
 \tilde\omega_n(x_1, \dots, x_{n-1}, m))\nonumber.
\end{eqnarray}
This ends the proof.
 \end{proof}
\begin{corollary}
 Let $(L, [\cdot, \dots, \cdot], \varepsilon, \alpha)$ be an $n$-Hom-Lie color algebra and $\beta : L\rightarrow L$ be an endomorphism.
Then $(L, \{\cdot, \dots, \cdot\}_i, \alpha)$, with
$$\{\cdot, \dots, \cdot\}_i=[\beta, \dots, \beta,\underbrace{id}_i, \beta, \dots, \beta], \quad i=1, 2, \dots, n,$$
is an $n$-Hom-Lie color module.
\end{corollary}

\begin{corollary}
 Let $(L, [\cdot, \dots, \cdot], \varepsilon, \alpha)$ be a multiplicative $n$-Hom-Lie color algebra.
Then, for nay $k\geq 1$, $(L, \{\cdot, \dots, \cdot\}_i^k, \alpha)$ is an $n$-Hom-Lie color module, with
$$\{\cdot, \dots, \cdot\}_i^{k}=[\alpha^k, \dots, \alpha^k,\underbrace{id}_i, \alpha^k, \dots, \alpha^k], \quad i=1, 2, \dots, n.$$
\end{corollary}
We end this section by giving some results for trivial gradation i.e. $G=\{e\}$.

\begin{proposition}
 Let $(M, \alpha_M, \omega_i)$ be a module over the $n$-Hom-Lie algebra \\ $(L, [\cdot, \dots, \cdot], \alpha_L)$. Consider the direct sum of linear spaces
$A=L\oplus M$. Let' us define on $A$ the bracket
\begin{itemize}
 \item $\{x_1, \dots, x_n\}=[x_1, \dots, x_n]$,
\item $\{x_1, \dots, x_{i-1}, \dots, m, x_{i+1}, \dots, x_n\}=\omega_i(x_1, \dots, x_{i-1}, \dots, m, x_{i+1}, \dots, x_n)$,
\item $\{x_1, \dots, x_{i}, \dots, x_{j}, \dots, x_n\}=0$, whenever $x_i, x_j\in M$.
\end{itemize}
Then $(A, \alpha_A=\alpha_L+\alpha_M)$ is an $n$-Hom-Lie algebra.
\end{proposition}

\begin{proposition}
 Let $(M_1, \alpha_M^1, \omega_i^1)$ and $(M_2, \alpha_M^2, \omega_i^2)$ be two modules over the $n$-Hom-Lie algebra $(L, [\cdot, \dots, \cdot], \alpha)$.
Then  $(M, \alpha_M, \omega_i)$ is an $n$-Hom-Lie module with
$$M=M_1\oplus M_2, \quad \alpha_M=\alpha_M^1\oplus \alpha_M^2\quad\mbox{and}\quad\omega_i=\omega_i^1\oplus\omega_i^2.$$
\end{proposition}


\section{Generalized derivation of color Hom-algebras and their color Hom-subalgebras}
\label{BakayokoSilvestrov:sec:genderivationsnhomcoloralg}
This section is devoted to generalized derivation of color Hom-algebras and their color Hom-subalgebras.

Here we give a more general definition of derivation, centroid and related objects.
\begin{definition}
For any $k\geq 1$, we call $D\in End(L)$ an $\alpha^k$-derivation \index{Derivation!n-Hom-Lie color algebra} of degree $d$ of the multiplicative $n$-Hom-Lie color algebra
$(L, [\cdot, \dots, \cdot], \varepsilon, \alpha)$ if
\begin{eqnarray}
&& \alpha\circ D=D\circ\alpha, \\
&& D([x_1, \dots, x_n])= \\
&& \sum_{i=1}^n\varepsilon(d, X_{i})[\alpha^k(x_1), \dots, \alpha^k(x_{i-1}), D(x_i), \alpha^k(x_{i+1}), \dots, \alpha^k(x_n)].
\label{BakayokoSilvestrov:der} \nonumber
\end{eqnarray}
\end{definition}

\begin{example}[\cite{ChapBaSilnhomliecolor:IBLaplacehomLiequasibialg}]
 The Laplacian \index{Laplacian!Hom-Lie quasi-bialgebra} of any Hom-Lie quasi-bialgebra $(\mathcal{G}, \mu, \gamma, \phi, \alpha)$ is an \\ $\alpha^2$-derivation of degree $0$
 of $(\Lambda\mathcal{G}, [\cdot, \cdot]^{\mu, \alpha})$, i.e.
\begin{eqnarray}
 L([X, Y ]^{\mu, \alpha} )= [L(X), \alpha^2(Y)]^{\mu, \alpha} + [\alpha^2(X), L(Y)]^{\mu, \alpha},\quad\forall X , Y \in \Lambda\mathcal{G}.
\quad\label{BakayokoSilvestrov:l2}
\end{eqnarray}
\end{example}

\begin{example}
Now, Let $(L, [\cdot, \dots, \cdot], \varepsilon, \alpha)$ be a multiplicative $n$-Hom-Lie color algebra.
For any  homogeneous elements $x_1, \dots, x_{n-1}$ of $L$ and any integer $k\geq 1$, one defines the adjoint action of
  $\Lambda L$ on $L$ by
\begin{multline*}
ad^{[\cdot, \dots, \cdot], \alpha^{k}}_{x_1, \dots, x_{n-1}}([y_1, y_2, \dots, y_n])
:= \\ \displaystyle\sum_{i=1}^k\varepsilon(X, Y_{i})[\alpha^k(y_1), \dots \alpha^k(y_{i-1}),
 [x_1, \dots, x_{n-1}, y_i], \alpha^k(y_{i+1}), \dots, \alpha(y_n)],
\end{multline*}
for any $y_1,\dots,y_n\in\mathcal{H}(L)$. Then $ad^{[\cdot, \dots, \cdot], \alpha^{k}}_{x_1, \dots, x_{n-1}}$ is an $\alpha^k$-derivation of $L$ of
degree $X$. We call $ad^{[\cdot, \dots, \cdot], \alpha^{k}}_{x_1, \dots, x_{n-1}}$ an inner $\alpha^k$-derivation. Denote by
$Inn(L)=\oplus_{k\geq-1}Inn_{\alpha^k}(L)$ the space of all inner $\alpha^k$-derivation.
\end{example}

The following proposition is proved by a straightforward computation.
\begin{proposition}
 Let $D$ be an $\alpha^k$-derivation of an $n$-Hom-Lie color algebra $L$ and $\beta : L\rightarrow L$ an even endomorphism of $L$ such that
$D\circ \beta=\beta\circ D$. Then, for any non-negative integer $s$, $\Delta_s=D\circ\beta^s : L\rightarrow L$ is a $(\beta^s\alpha^k)$-derivation.
\end{proposition}
\begin{corollary}
 If $D$ is an $\alpha^{k}$-derivation of an $n$-Hom-Lie color algebra $L$. Then $\Delta_s$ is an $\alpha^{k+s}$-derivation of $L$.
\end{corollary}

We denote the set of $\alpha^k$-derivations of the multiplicative $n$-Hom-Lie color algebras $L$ by $Der_{\alpha^k}(L)$.
For any $D\in Der_{\alpha^k}(L)$ and $D'\in Der_{\alpha^k}(L)$, let us define their color commutator $[D, D']$ as usual by
$[D, D']=D\circ D'-\varepsilon(d, d')D'\circ D.$
\begin{lemma}
 For any $D\in Der_{\alpha^k}(L)$ and $D'\in Der_{\alpha^k}(L)$,
$$[D, D']\in Der_{\alpha^{k+s}}(L).$$
\end{lemma}
Denote by $Der(L)=\oplus_{k\geq -1}Der_{\alpha^k}(L)$.
\begin{proposition}
$(Der(L), [\cdot, \cdot], \omega)$ is a Hom-Lie color algebra, with $\omega (D)=D\circ \alpha$.
\end{proposition}

\begin{definition}
 An endomorphism $D$ of degree $d$ of a multiplicative $n$-Hom-Lie color algebra $(L, [\cdot, \dots, \cdot], \varepsilon, \alpha)$ is called a generalized
$\alpha^k$-derivation \index{Derivation!generalised!multiplicative n-Hom-Lie color algebra} if there exist linear mappings $D', D'', \dots, D^{(n-1)}, D^{(n)}$ of degree $d$ such that for any $x_1,\dots,x_n\in\mathcal{H}(L)$:
\begin{eqnarray}
&& D\circ\alpha=\alpha\circ D\;\;\mbox{and}\;\; D^{(i)}\circ\alpha=\alpha\circ D^{(i)},\\
&& \begin{array}{l}
D^{(n)}([x_1, \dots, x_n])= \\
\displaystyle \sum_{i=1}^n\varepsilon(d, X_{i})[\alpha^k(x_1), \dots, \alpha^k(x_{i-1}), D^{(i-1)}(x_i), \alpha^k(x_{i+1}), \dots, \alpha^k(x_n)].
\end{array}
\end{eqnarray}
An $(n+1)$-tuple $(D, D', D'', \dots, D^{(n-1)}, D^{(n)})$ is called an $(n+1)$-ary $\alpha^k$-derivation.
\end{definition}
The set of generalized $\alpha^k$-derivation is denoted by $GDer_{\alpha^k}$. Set $$GDer(L)=\oplus_{k\geq -1}GDer_{\alpha^k}(L).$$

\begin{definition}
Let $(L, [\cdot, \dots, \cdot], \varepsilon, \alpha)$ be a multiplicative $n$-Hom-Lie color algebra. A linear mapping $D\in End(L)$ is said to be an
$\alpha^k$-quasiderivation of degree $d$  \index{Quasiderivation!multiplicative n-Hom-Lie color algebra} if there exists a $D'\in End(L)$ of degree $d$ such that
\begin{eqnarray}
&& D^{'}([x_1, \dots, x_n]) = \\
&& \sum_{i=1}^n\varepsilon(d, X_{i})[\alpha^k(x_1), \dots, \alpha^k(x_{i-1}), D(x_i), \alpha^k(x_{i+1}), \dots, \alpha^k(x_n)]
\nonumber
\end{eqnarray}
for all $x_1,\dots,x_n\in\mathcal{H}(L)$
\end{definition}
We call $D'$ the endomorphism associated to the $\alpha^k$-quasiderivation $D$.
The set of \\
$\alpha^k$-quasiderivations will be denoted $QDer(L)$. Set $QDer(L)=\oplus_{k\geq -1}QDer_{\alpha^k}(L)$.

\begin{definition}
 Let $(L, [\cdot, \dots, \cdot], \varepsilon, \alpha)$ be a multiplicative $n$-Hom-Lie color algebra. The set $C_{\alpha^k}(L)$ consisting of linear mapping
 $D$ of degree $d$ with
the property
\begin{eqnarray}
&& D([x_1, \dots, x_n])= \\
&& \varepsilon(d, X_{i})[\alpha^k(x_1), \dots, \alpha^k(x_{i-1}), D(x_i), \alpha^k(x_{i+1}), \dots, \alpha^k(x_n)]
\nonumber
\end{eqnarray}
for all $x_1,\dots,x_n\in\mathcal{H}(L)$, is called the $\alpha^k$-centroid of $L$. \index{Centroid!multiplicative n-Hom-Lie color algebra}
\end{definition}
We recover the definition of the centroid when $k=0$.  \index{Centroid}

\begin{definition}
 Let $(L, [\cdot, \dots, \cdot], \varepsilon, \alpha)$ be a multiplicative $n$-Hom-Lie color algebra. The set $QC_{\alpha^k}(L)$ consisting of linear mapping
 $D$ of degree $d$ with
the property
\begin{eqnarray}
&& [D(x_1), \dots, x_n]= \\
&& \varepsilon(d, X_{i})[\alpha^k(x_1), \dots, \alpha^k(x_{i-1}), D(x_i), \alpha^k(x_{i+1}), \dots, \alpha^k(x_n)],
\nonumber
\end{eqnarray}
for all $x_1,\dots,x_n\in\mathcal{H}(L)$, is called the $\alpha^k$-quasicentroid of $L$. \index{Quasicentroid!multiplicative n-Hom-Lie color algebra}
\end{definition}

\begin{definition}
  Let $(L, [\cdot, \dots, \cdot], \varepsilon, \alpha)$ be a multiplicative $n$-Hom-Lie color algebra. The set $ZDer_{\alpha^k}(L)$ consisting of linear mappings $D$ of degree $d$, such that for all $x_1,\dots,x_n\in\mathcal{H}(L)$:
\begin{eqnarray}
&& D([x_1, \dots, x_n])= \\
&& \varepsilon(d, X_i)[\alpha^k(x_1), \dots, \alpha^k(x_{i-1}), D(x_i), \alpha^k(x_{i+1}), \dots, \alpha^k(x_n)]=0, \nonumber \\
&& \hspace{6cm} i=1, 2, \dots, n, \label{BakayokoSilvestrov:cd} \nonumber
\end{eqnarray}
is called the set of central $\alpha^k$-derivations of $L$. \index{Derivation!central!multiplicative n-Hom-Lie color algebra}
\end{definition}

It is easy to see that
$$ZDer(L)\subseteq Der(L)\subseteq QDer(L)\subseteq GDer(L)\subseteq End(L).$$

\begin{proposition}
 Let $(L, [\cdot, \dots, \cdot], \varepsilon, \alpha)$ be a multiplicative $n$-Hom-Lie color algebra.
\begin{enumerate}
 \item [1)] $GDer(L), QDer(L)$, $C(L)$ are color Hom-subalgebras of
 $(End (L), [\cdot, \cdot], \omega)${\rm :}
\begin{enumerate}
\item[1a)] $\omega(GDer(L))\subseteq GDer(L)$ and $[GDer(L), GDer(L)]\subseteq GDer(L)$.
\item[2b)] $\omega(QDer(L))\subseteq QDer(L)$ and $[QDer(L), QDer(L)]\subseteq QDer(L)$.
\item[3c)] $\omega(C(L))\subseteq C(L)$ and $[C(L), C(L)]\subseteq C(L)$.
\end{enumerate}
\item [2)] $ZDer(L)$ is a color Hom-ideal of $Der(L)${\rm :} \\
$\omega(ZDer(L))\subseteq ZDer(L)$ and $[ZDer(L), Der(L)]\subseteq ZDer(L)$.
\end{enumerate}
\end{proposition}
\begin{proof}
1a) Let us prove that if $D\in GDer(L)$, then $\omega(D)\in GDer(L)$.

 For any $x_1,\dots,x_n\in\mathcal{H}(L)$,
\begin{eqnarray}
 &&(\omega(D^{(n)}))([x_1, \dots, x_n])=(D^{(n)}\circ\alpha)([x_1, \dots, x_i, \dots, x_n]) \nonumber \\
&&=D^{(n)}([\alpha(x_1), \dots, \alpha(x_i), \dots, \alpha(x_n)])\nonumber\\
&&=\sum_{i=1}^n\varepsilon(d, X_{i})[\alpha^{k+1}(x_1), \dots, \alpha^{k+1}(x_{i-1}), D^{(i-1)}\alpha(x_i), \alpha^{k+1}(x_{i+1}), \dots,
\alpha^{k+1}(x_n)]\nonumber\\
&&=\sum_{i=1}^n\varepsilon(d, X_{i})[\alpha^{k+1}(x_1), \dots, \alpha^{k+1}(x_{i-1}), ( D^{(i-1)}\circ\alpha)(x_i), \nonumber \\
&& \hspace{6cm} \alpha^{k+1}(x_{i+1}), \dots,\alpha^{k+1}(x_n)]\nonumber\\
&&=\sum_{i=1}^n\varepsilon(d, X_{i})[\alpha^{k+1}(x_1), \dots, \alpha^{k+1}(x_{i-1}), \omega(D^{(i-1)})(x_i), \alpha^{k+1}(x_{i+1}), \dots,
\alpha^{k+1}(x_n)]\nonumber.
\end{eqnarray}
This means that $\omega(D)$ is an $\alpha^{k+1}$-derivation i.e. $\omega(D)\in GDer(L)$.
Now let $D_1\in GDer_{\alpha^{k}}(L)$ and $D_2\in GDer_{\alpha^{s}}(L)$, we have
\begin{eqnarray}
 &&(D_2^{(n)}D_1^{(n)})([x_1, \dots, x_n])=D_2^{(n)}(D_1^{(n)}([x_1, \dots, x_n]))=\nonumber\\
&&=\sum_{i=1}^n\varepsilon(d_1, X_{i})D_2^{(n)}([\alpha^{k}(x_1), \dots, \alpha^{k}(x_{i-1}), D_1^{(i-1)}(x_i), \alpha^{k}(x_{i+1}), \dots,
\alpha^{k}(x_n)])\nonumber\\
&&=\sum_{i=1}^n\sum_{j<i}^n\varepsilon(d_1, X_{i})\varepsilon(d_2, X_{j})
([\alpha^{k+s}(x_1), \dots, D_2^{(j-1)}(x_j),\dots, \alpha^{k+s}(x_{i-1}), D_1^{(i-1)}\alpha^{s}(x_i),
\nonumber \\
&& \hspace{7cm}
\alpha^{k+s}(x_{i+1}), \dots,
\alpha^{k+s}(x_n)])\nonumber\\
&&+\sum_{i=1}^n\varepsilon(d_1+d_2, X_{i})
([\alpha^{k+s}(x_1), \dots,  \alpha^{k+s}(x_{i-1}), D_2^{(i-1)}D_1^{(i-1)}(x_i), \nonumber \\
&& \hspace{7cm}
\alpha^{k+s}(x_{i+1}), \dots,
\alpha^{k}(x_n)])\nonumber\\
&&+\sum_{i=1}^n\sum_{j>i}^n\varepsilon(d_1, X_{i})\varepsilon(d_2, d_1+X_{j})
([\alpha^{k+s}(x_1), ,\dots, \alpha^{k+s}(x_{i-1}), D_1^{(i-1)}\alpha^{s}(x_i), \nonumber \\
&& \hspace{5cm}
\alpha^{k+s}(x_{i+1}), \dots, D_2^{(j-1)}(x_j), \dots,
\alpha^{k+s}(x_n)])\nonumber.
\end{eqnarray}
It follows that
\begin{eqnarray}
 &&([D_1^{(n)}, D_2^{(n)}])([x_1, \dots, x_n])=(D_1^{(n)}D_2^{(n)}-\varepsilon(d_1, d_2)D_2^{(n)}(D_1^{(n)})([x_1, \dots, x_n]))=\nonumber\\
&&=\sum_{i=1}^n\varepsilon(d_1+d_2, X_{i})
([\alpha^{k+s}(x_1), \dots,  \alpha^{k+s}(x_{i-1}),
\nonumber \\
&& \hspace{1.5cm} (D_1^{(i-1)}D_2^{(i-1)}-\varepsilon(d_1, d_2)D_2^{(i-1)}D_1^{(i-1)}](x_i), \alpha^{k+s}(x_{i+1}), \dots,
\alpha^{k+s}(x_n)])\nonumber\\
&&=\sum_{i=1}^n\varepsilon(d_1+d_2, X_{i})
([\alpha^{k+s}(x_1), \dots,  \alpha^{k+s}(x_{i-1}),
\nonumber \\
&& \hspace{1.5cm}
[D_1^{(i-1)}, D_2^{(i-1)}](x_i), \alpha^{k+s}(x_{i+1}), \dots,
\alpha^{k+s}(x_n)])\nonumber.
\end{eqnarray}
Thus we obtain that $[D_1, D_2]\in GDer_{\alpha^{k+s}}(L)$.
\begin{enumerate}
\item[1b)] That $QDer(L)$ is a color Hom-subalgebra of $(End (L), [\cdot, \cdot], \omega)$ is proved in the similar way.
\item[1c)]  Let $D_1\in C_{\alpha^{k}}(L)$ and $D_2\in C_{\alpha^{s}}(L)$. Then
\begin{eqnarray}
&& \omega(D_1)([x_1, x_2, \dots, x_n])=\alpha D_1([x_1, x_2, \dots, x_n])\nonumber\\
&&=\varepsilon(d_1, X_i)\alpha([\alpha^k(x_1), \alpha^k(x_2), \dots, D_1(x_i), \dots, \alpha^k(x_n])\nonumber\\
&&=\varepsilon(d_1, X_i)[\alpha^{k+1}(x_1), \alpha^{k+1}(x_2), \dots, D_1(x_i), \dots, \alpha^{k+1}(x_n])\nonumber.
\end{eqnarray}
Thus $\omega(D)\in C_{\alpha^{k+1}}(L)$. Moreover,
\begin{eqnarray}
&& [D_1, D_2]([x_1, \dots, x_n])=D_1D_2([x_1, \dots, x_n])-
\varepsilon(d_1, d_2)D_2D_1([x_1, \dots, x_n])\nonumber\\
&&=\varepsilon(d_2, X_i)D_1[\alpha^k(x_1), \alpha^k(x_2), \dots, D_2(x_i), \dots, \alpha^k(x_n)]\nonumber\\
&&-\varepsilon(d_1, d_2)\varepsilon(d_1, X_i)D_2[\alpha^s(x_1), \alpha^s(x_2), \dots, D_1(x_i), \dots, \alpha^s(x_n)]\nonumber\\
&&=\varepsilon(d_1+d_2, X_i)[\alpha^{k+s}(x_1), \alpha^{k+s}(x_2), \dots,D_1 D_2(x_i), \dots, \alpha^{k+s}(x_n)]\nonumber\\
&&-\varepsilon(d_1+d_2, X_i)[\alpha^{k+s}(x_1), \alpha^{k+s}(x_2), \varepsilon(d_1, d_2)\dots,D_2 D_1(x_i), \dots, \alpha^{k+s}(x_n)]\nonumber\\
&&=\varepsilon(d_1+d_2, X_i)[\alpha^{k+s}(x_1), \alpha^{k+s}(x_2), \dots, [D_1 D_2](x_i), \dots, \alpha^{k+s}(x_n)]\nonumber.
\end{eqnarray}
So, $[D_1, D_2]\in C_{\alpha^{k+s}}(L)$ and finally $[D_1, D_2]\in C(L)$.
\item[2)] By the same method as previously one can show that $\omega(D)\in ZDer_{\alpha^{k+1}}(L)$ and \\
 $[D_1, D_2]\in ZDer_{\alpha^{k+s}}(L)$, where $D_1\in ZDer_{\alpha^{k}}(L)$ and $D_2\in Der_{\alpha^{s}}(L)$.
\end{enumerate}
\end{proof}

\begin{proposition}
 Let $(L, [\cdot, \dots, \cdot], \varepsilon, \alpha)$ be a multiplicative $n$-Hom-Lie color algebra.
\begin{enumerate}
\item [1)]
If $\varphi\in C(L)$ and $D\in Der(L)$, then $\varphi D$ is a derivation i.e.
$$C(L)\cdot Der(L)\subseteq Der(L).$$
\item [2)] Any element of centroid is a quasiderivation i.e.
 $$C(L)\subseteq QDer(L).$$
\end{enumerate}
\end{proposition}

\begin{proof}
1) For any $x_1, \dots, x_n\in \mathcal{H}(L)$,
\begin{eqnarray}
 \varphi D([x_1, \dots, x_n])&=&\sum_{i=1}^n\varepsilon(d, X_i)\varphi([\alpha^k(x_1), \dots, D(x_i), \dots, \alpha^k(x_n)])\nonumber\\
&=&\sum_{i=1}^n\varepsilon(d, X_i)\varepsilon(\varphi, X_i)[\alpha^{k+s}(x_1), \dots, \varphi D(x_i), \dots, \alpha^{k+s}(x_n)])\nonumber\\
&=&\sum_{i=1}^n\varepsilon(d+\varphi, X_i)[\alpha^{k+s}(x_1), \dots, \varphi D(x_i), \dots, \alpha^{k+s}(x_n)])\nonumber.
\end{eqnarray}
Thus $\varphi D$ is an $\alpha^{k+s}$-derivation of degree $d+\varphi$.

2) Let $D$ be an $\alpha^{k}$-centroid, then for any $x_1, \dots, x_n\in \mathcal{H}(L),$
\begin{eqnarray}
 D([x_1, \dots, x_n])=\varepsilon(d, X_i)[\alpha^k(x_1), \dots D(x_i), \dots, \alpha^k(x_n)], i=1, 2, \dots, n.
\end{eqnarray}
It follows that
\begin{eqnarray}
 \sum_{i=1}^n\varepsilon(d, X_i)[\alpha^k(x_1), \dots D(x_i), \dots, \alpha^k(x_n)]=nD([x_1, \dots, x_n]).
\end{eqnarray}
It suffises to take $D'=nD$.
 \end{proof}

\begin{lemma}\label{BakayokoSilvestrov:lm}
 Let $(L, [\cdot, \dots, \cdot], \varepsilon, \alpha)$ be a multiplicative $n$-Hom-Lie color algebra. Then
\begin{enumerate}
\item [1)] The $\varepsilon$-commutator of two elements of quasicentroid is a quasiderivation i.e.
 $$[QC(L), QC(L)]\subseteq QDer(L).$$
\item [2)] $QDer(L)+QC(L)\subseteq GDer(L)$.
\end{enumerate}
\end{lemma}
\begin{proof}
 For any $x_1, x_2, \dots, x_n\in\mathcal{H}(L)$,
\begin{enumerate}
 \item [1)] let $D_1\in QC_{\alpha^k}(L)$ and $D_2\in QC_{\alpha^s}(L)$. We have, on the one hand
\begin{eqnarray}
&&[D_1D_2(x_1), \alpha^{k+s}(x_2), \dots, \alpha^{k+s}(x_n)]\nonumber\\
&&=\varepsilon(D_1, D_2+X_i)[D_2(\alpha^{k}(x_1)), \alpha^{k+s}(x_2), \dots, D_1(\alpha^{s}(x_i)), \dots, \alpha^{k+s}(x_n)]\nonumber\\
&&=\varepsilon(D_1, D_2+X_i)\varepsilon(D_2, X_i)[\alpha^{k+s}(x_1), \dots, D_2D_1(x_i), \dots, \alpha^{k+s}(x_n)]\nonumber\\
&&=\varepsilon(D_1, D_2)\varepsilon(D_1+D_2, X_i)[\alpha^{k+s}(x_1), \dots, D_2D_1(x_i), \dots, \alpha^{k+s}(x_n)]\nonumber.
\end{eqnarray}
On the other hand,
\begin{eqnarray}
 &&[D_1D_2(x_1), \alpha^{k+s}(x_2), \dots, \alpha^{k+s}(x_n)]=\nonumber\\
&&=\varepsilon(D_1, D_2+x_1)[D_2(\alpha^{k}(x_1)), D_1(\alpha^{s}(x_2)), \dots, \alpha^{k+s}(x_i), \dots, \alpha^{k+s}(x_n)]\nonumber\\
&&=\varepsilon(D_1, D_2+x_1)\varepsilon(D_2, D_1+X_i)
\nonumber \\
&& \hspace{3cm}
[\alpha^{k+s}(x_1), D_1(\alpha^{s}(x_2)), \dots, D_2(\alpha^{k}(x_i)), \dots, \alpha^{k+s}(x_n)]\nonumber\\
&&=\varepsilon(D_1, x_1)\varepsilon(D_2, X_i)\varepsilon(x_1, D_1)
\nonumber \\
&&
[D_1(\alpha^{s}(x_1)), \alpha^{k+s}(x_2), \dots, D_2(\alpha^{k}(x_i)), \dots, \alpha^{k+s}(x_n)]\nonumber\\
&&=\varepsilon(D_2, X_i)\varepsilon(D_1, X_i)[\alpha^{k+s}(x_1), \dots, D_1D_2(x_i), \dots, \alpha^{k+s}(x_n)]\nonumber\\
&&=\varepsilon(D_1+D_2, X_i)[\alpha^{k+s}(x_1), \dots, D_1D_2(x_i), \dots, \alpha^{k+s}(x_n)],\nonumber
\end{eqnarray}
and so
\begin{eqnarray}
&&\varepsilon(D_1+D_2, X_i)[\alpha^{k+s}(x_1), \dots, [D_1, D_2](x_i), \dots, \alpha^{k+s}(x_n)]=\nonumber\\
&&=\varepsilon(D_1+D_2, X_i)\Big([\alpha^{k+s}(x_1), \dots, D_1D_2(x_i), \dots, \alpha^{k+s}(x_n)]\nonumber\\
&&\quad-\varepsilon(D_1, D_2)[\alpha^{k+s}(x_1), \dots, D_2D_1(x_i), \dots, \alpha^{k+s}(x_n)]\Big)=0.\nonumber
\end{eqnarray}
It follows that
 \begin{eqnarray}
  \sum_{i=1}^n\varepsilon(D_1+D_2, X_i)[\alpha^{k+s}(x_1), \dots, [D_1, D_2](x_i), \dots, \alpha^{k+s}(x_n)]=0\nonumber.
 \end{eqnarray}
Therefore $D'\equiv 0$, and $[D_1, D_2]\in QDer(L)$.
\item [2)] Let $D_1\in QDer_{\alpha^k}(L)$ and $D_2\in QC_{\alpha^k}(L)$ with $|D_1|=|D_2|$. \\
Then there exists $D'_1\in End(L)$ such that
\begin{multline}
D'_1([x_1, \dots, x_n])
=\sum_{i=1}^n \varepsilon(D_1, X_i)[\alpha^{k}(x_1), \dots, D_1(x_i), \dots, \alpha^{k}(x_n)]\nonumber\\
=[D_1(x_1), \alpha^{k}(x_2), \dots, \alpha^{k}(x_n)]+\varepsilon(D_1, x_1)[\alpha^{k}(x_1), D_1(x_2), \dots, \alpha^{k}(x_n)]\nonumber\\
+\sum_{i=3}^n\varepsilon(D_1, X_i)[\alpha^{k}(x_1), \dots, D_1(x_i), \dots, \alpha^{k}(x_n)]\nonumber\\
=[(D_1+D_2)(x_1), \alpha^{k}(x_2), \dots, \alpha^{k}(x_n)]-[D_2(x_1), \alpha^{k}(x_2), \dots, \alpha^{k}(x_n)]\nonumber\\
+\varepsilon(D_1, x_1)[\alpha^{k}(x_1), D_1(x_2), \dots, \alpha^{k}(x_n)]
\nonumber \\
+ \sum_{i=3}^n\varepsilon(D_1, X_i)[\alpha^{k}(x_1), \dots, D_1(x_i), \dots, \alpha^{k}(x_n)]\nonumber\\
=[(D_1+D_2)(x_1), \alpha^{k}(x_2), \dots, \alpha^{k}(x_n)]-\varepsilon(D_2, x_1)[\alpha^{k}(x_1), D_2(x_2), \dots, \alpha^{k}(x_n)]\nonumber\\
+\varepsilon(D_1, x_1)[\alpha^{k}(x_1), D_1(x_2), \dots, \alpha^{k}(x_n)]
\nonumber \\
 \qquad + \sum_{i=3}^n\varepsilon(D_1, X_i)[\alpha^{k}(x_1), \dots, D_1(x_i), \dots, \alpha^{k}(x_n)]\nonumber\\
=[(D_1+D_2)(x_1), \alpha^{k}(x_2), \dots, \alpha^{k}(x_n)]
\nonumber\\
+\varepsilon(D_2, x_1)[\alpha^{k}(x_1), (D_1-D_2)(x_2), \dots, \alpha^{k}(x_n)]
\nonumber\\
+\sum_{i=3}^n\varepsilon(D_1, X_i)[\alpha^{k}(x_1), \dots, D_1(x_i), \dots, \alpha^{k}(x_n)]\nonumber.
\end{multline}
\end{enumerate}
The conclusion follows by taking
$$D^{(n)}=D'_1,\quad D=D_1+D_2,\quad D'=D_1-D_2,\quad D^{(i)}=D_1,\quad 2\leq i\leq n-1.$$
This proved that $D_1+D_2\in GDe(L)$.
 \end{proof}

\begin{proposition}
If $(L, [\cdot, \dots, \cdot], \varepsilon, \alpha)$ is a multiplicative $n$-Hom-Lie color algebra, then
$$QC(L)+[QC(L), QC(L)]$$
 is a color Hom-subalgebra of $GDer(L)$.
\end{proposition}
 \begin{proof}
 It follows from Lemma \ref{BakayokoSilvestrov:lm} by using the same arguments as in Proposition 2.4  in \cite{ChapBaSilnhomliecolor:KP}.
 \end{proof}

%
%

\begin{proposition}
 Let $(L, [\cdot, \dots, \cdot], \varepsilon, \alpha)$ is a multiplicative $n$-Hom-Lie color algebra such that $\alpha$ be a surjective mapping, then
$[C(L), QC(L)]\subseteq Hom(L, Z(L))$. Moreover, if $Z(L)=\{0\}$, then $[C(L), QC(L)]=\{0\}$.
\end{proposition}
\begin{proof}
 Let $D_1\in C_{\alpha^k}(L)$,  $D_2\in QC_{\alpha^s}(L)$ and $x_1, \dots, x_n\in \mathcal{H}(L)$. Since $\alpha$ is surjective, for
any $y'_i\in L$, there exists $y_i\in L$ such that $y'_i=\alpha^{k+s}(y_i), i= 2, \dots, n$. Thus
\begin{eqnarray}
&& [[D_1, D_2](x_1), y'_2, \dots, y'_n]=\nonumber\\
&&=[[D_1, D_2](x_1), \alpha^{k+s}(y_2), \dots, \alpha^{k+s}(y_n)]\nonumber\\
&&=[D_1D_2(x_1), \alpha^{k+s}(y_2), \dots, \alpha^{k+s}(y_n)]
\nonumber\\
&& \qquad
-\varepsilon(d_1, d_2)[D_2D_1(x_1), \alpha^{k+s}(y_2), \dots, \alpha^{k+s}(y_n)]\nonumber\\
&&=D_1([D_2(x_1), \alpha^{s}(y_2), \dots, \alpha^{s}(y_n)])
\nonumber\\
&& \qquad
-\varepsilon(d_1, d_2)\varepsilon(d_2,x_1+ d_1)[D_1\alpha^{s}(x_1), D_2\alpha^{k}(y_2), \dots, \alpha^{k+s}(y_n)]\nonumber\\
&&=D_1([D_2(x_1), \alpha^{k+s}(y_2), \dots, \alpha^{k+s}(y_n)])
\nonumber\\
&& \qquad
-\varepsilon(d_2, x_1)D_1[\alpha^{s}(x_1), \alpha^{s}D_2(y_2), \dots, \alpha^{s}(y_n)]\nonumber\\
&&=D_1\Big([D_2(x_1), \alpha^{k+s}(y_2), \dots, \alpha^{k+s}(y_n)])
\nonumber\\
&& \qquad
-\varepsilon(d_2, x_1)[\alpha^{s}(x_1), \alpha^{s}D_2(y_2), \dots, \alpha^{s}(y_n)]\Big)\nonumber\\
&&=D_1\Big([D_2(x_1), \alpha^{k+s}(y_2), \dots, \alpha^{k+s}(y_n)])
-[D_2(x_1), \alpha^{s}(y_2), \dots, \alpha^{s}(y_n)]\Big)=0.\nonumber
\end{eqnarray}
Hence, $[D_1, D_2](x_1)\in Z(L)$, and $[D_1, D_2]\in Hom(L, Z(L))$. Furthermore, if $Z(L)=\{0\}$, we know that $[C(L), QC(L)]=\{0\}$.
 \end{proof}

\begin{proposition}
 Let $(L, [\cdot, \dots, \cdot], \varepsilon, \alpha)$ is a multiplicative $n$-Hom-Lie color algebra with surjective twisting $\alpha$
  and $H$ be a graded subset of $L$. Then
\begin{enumerate}
 \item [i)] $Z_L(H)$ is invariant under $C(L)$.
\item [ii)] Every perfect color Hom-ideal of $L$ is invariant under $C(L)$.
\end{enumerate}
\end{proposition}
\begin{proof}
\begin{enumerate}
 \item [i)]
 For any $\varphi\in C(L)$ and $x\in Z_L(H)$, by \eqref{BakayokoSilvestrov:z}, we have
\begin{eqnarray}
 0=\varphi([x, H, L, \dots, L])=[\varphi(x), \alpha^k(H), \alpha^k(L), \dots, \alpha^k(L)]=[\varphi(x), H, L, \dots, L].\nonumber
\end{eqnarray}
Therefore $\varphi(x)\in Z_L(H)$, which implies that $Z_L(H)$ is invariant under $C(L)$.
\item [ii)] Let $H$ be a perfect color Hom-ideal of $L$. Then $H^1=H$, and so for any $x\in H$ there exist $x_1^i, x_2^i, \dots, x_n^i\in H$ with
$0<i<\infty$ such that $x=\sum_i[x_1^i, x_2^i, \dots, x_n^i]$. If $\varphi\in C(L)$, then
\begin{eqnarray*}
\varphi(x) &=& \varphi(\sum_i[x_1^i, x_2^i, \dots, x_n^i])=\sum_i\varphi([x_1^i, x_2^i, \dots, x_n^i])
\nonumber \\
&=&
\sum_i[\varphi(x_1^i), \alpha^k(x_2^i), \dots, \alpha^k(x_n^i)])\in H.
\end{eqnarray*}
This shows that $H$ is invariant under $C(L)$.

\end{enumerate}
\end{proof}

\begin{proposition}\label{BakayokoSilvestrov:dc1}
 If the characteristic of $\mathbb{K}$ is $0$ or not a factor of $n-1$. Then
$$ZDer(L)=C(L)\cap Der(L).$$
\end{proposition}
\begin{proof}
 If $\varphi\in C(L)\cap Der(L)$, then by \eqref{BakayokoSilvestrov:der}  we have
$$\varphi([x_1, \dots, x_n])=\sum_{i=1}^n\varepsilon(d, X_i)[\alpha^k(x_1), \dots, \varphi(x_i), \dots, \alpha^k(x_n)],$$
and by \eqref{BakayokoSilvestrov:cd}, for $i=1, 2, \dots, n$,
$$\varepsilon(d, X_i)[\alpha^k(x_1), \dots, \varphi(x_i), \dots, \alpha^k(x_n)]=\varphi([x_1, \dots, x_n]).$$
Thus
$\varphi([x_1, \dots, x_n])=n \varphi([x_1, \dots, x_n])$.
The characteristic of  $\mathbb{K}$ being $0$ or not a factor of $n-1$. We have
$$0=\varphi([x_1, \dots, x_n])=\varepsilon(d, X_i)[\alpha^k(x_1), \dots, \varphi(x_i), \dots, \alpha^k(x_n)], \quad i=1, 2, \dots, n.$$
Which means that $\varphi\in ZDer(L)$.\\
Conversly, let $\varphi\in ZDer(L)$, Then
$$\varphi([x_1, \dots, x_n])=\varepsilon(d, X_i)[\alpha^k(x_1), \dots, \varphi(x_i), \dots, \alpha^k(x_n)]=0, \quad 1\leq i\leq n$$
and thus $\varphi\in C(L)\cap Der(L)$. Therefore $ZDer(L)=C(L)\cap Der(L)$.
 \end{proof}

\begin{proposition}\label{BakayokoSilvestrov:dc2}
 Let $L$ be an $n$-Hom-Lie color algebra. For any $D\in Der(L)$ and $\varphi\in C(L)$
\begin{enumerate}
 \item[1)] $Der(L)$ is contained in the normalizer of $C(L)$ in $End(L)$ i.e.
$$[Der(L), C(L)]\subseteq C(L).$$
\item [2)] $QDer(L)$ is contained in the normalizer of $QC(L)$ in $End(L)$ i.e.
$$[QDer(L), QC(L)]\subseteq QC(L).$$
\end{enumerate}
\end{proposition}
\begin{proof}
1) For any $D\in Der(L), \varphi\in C(L)$ and $x_1, x_2, \dots, x_n\in\mathcal{H}(L)$,
 \begin{eqnarray}
&& D\varphi([x_1, \dots, x_n])=D([\varphi(x_1), \alpha^{k}(x_2), \dots, \alpha^{k}(x_i), \dots, \alpha^{k}(x_n)])\nonumber\\
&&=[D\varphi(x_1), \alpha^{k+s}(x_2), \dots, \alpha^{k+s}(x_i), \dots, \alpha^{k+s}(x_n)]\nonumber\\
&& \quad +\sum_{i=2}^n\varepsilon(d, \varphi+ X_i)[\alpha^{s}\varphi(x_1), \alpha^{k+s}(x_2), \dots, \alpha^{k}D(x_i), \dots, \alpha^{k+s}(x_n)]\nonumber\\
&&=[D\varphi(x_1), \alpha^{k+s}(x_2), \dots, \alpha^{k+s}(x_i), \dots, \alpha^{k+s}(x_n)]\nonumber\\
&& \quad +\sum_{i=2}^n\varepsilon(d,\varphi+ X_i)\varepsilon(\varphi, X_i)([\alpha^{k+s}(x_1), \alpha^{k+s}(x_2), \dots, \varphi D(x_i),
\dots, \alpha^{k+s}(x_n)])\nonumber\\
&=& [D\varphi(x_1), \alpha^{k+s}(x_2), \dots, \alpha^{k+s}(x_i), \dots, \alpha^{k+s}(x_n)]\nonumber\\
&& \quad +\varepsilon(d,\varphi)\sum_{i=2}^n\varepsilon(d+\varphi, X_i)([\alpha^{k+s}(x_1), \alpha^{k+s}(x_2), \dots, \varphi D(x_i),
\dots, \alpha^{k+s}(x_n)])\nonumber\\
&=&[D\varphi(x_1), \alpha^{k+s}(x_2), \dots, \alpha^{k+s}(x_i), \dots, \alpha^{k+s}(x_n)]\nonumber\\
&&\quad +\varepsilon(d,\varphi)\Big(\varphi D[x_1, x_2, \dots, x_i, \dots, x_n]\nonumber\\
&&\quad-[\varphi D(x_1), \alpha^{k+s}(x_2), \dots, \alpha^{k+s}(x_i), \dots, \alpha^{k+s}(x_n)]   \Big)\nonumber.
 \end{eqnarray}
Then we get
\begin{eqnarray}
&& (D\varphi-\varepsilon(d, \varphi)\varphi D)([x_1, \dots, x_n])
 \nonumber \\
&&\qquad  =[(D\varphi-\varepsilon(d, \varphi)\varphi D)(x_1), \dots, \alpha^{k+s}(x_2),
 \dots, \alpha^{k+s}(x_i), \dots, \alpha^{k+s}(x_n)])\nonumber,
\end{eqnarray}
that is
$[D, \varphi]=D\varphi-\varepsilon(d, \varphi)\varphi D\in C(L)$. \\
2) It is proved by using a similar method.
 \end{proof}

\begin{proposition}
 Let $L$ be an $n$-Hom-Lie color algebra. For any $D\in Der(L)$ and $\varphi\in C(L)$
\begin{enumerate}
\item [1)] $D\varphi$ is contained in $C(L)$ if and only if $\varphi D$ is a central derivation of $L$.
\item [2)] $D\varphi$ is a derivation of $L$ if and only if $[D, \varphi]$ is a central derivation of $L$.
\end{enumerate}
\end{proposition}
\begin{proof}
1) From Proposition \ref{BakayokoSilvestrov:dc2}, $D\varphi$ is an element of $C(L)$ if and only if $\varphi D\in  Der(L)\cap C(L)$.
Thanks to Proposition \ref{BakayokoSilvestrov:dc1}, we get the result. \\
2) The conclusion follows from $1)$, Proposition \ref{BakayokoSilvestrov:dc1} and Proposition \ref{BakayokoSilvestrov:dc2}.
 \end{proof}

If $A$ is a commutative associative algebra and $L$ is an $n$-Hom-Lie color algebra, the $n$-Hom-Lie algebra $A\otimes L$ (Theorem \ref{BakayokoSilvestrov:tp})
 is called \index{Tensor product!n-Hom-Lie color algebra} the tensor product $n$-Hom-Lie color algebra of $A$ and $L$. For $f\in End(A)$ and $\varphi\in End(L)$ let
$f\otimes \varphi : A\otimes L\rightarrow A\otimes L$ be given by $f\otimes \varphi(a\otimes x)=f(a)\otimes\varphi(x)$, for $a\in A, x\in L$.
Then $f\otimes\varphi\in End(A\otimes L)$.

Recall that if $A$ is a commutative associative algebra, the centroid $C(A)$ of $A$ is by definition
$$C(A)=\{f\in End(A) \mid  f(ab)=f(a)b=af(b), \forall a, b\in A\}.$$
We now state the following proposition.
\begin{proposition}
With the above notation, we have
$$C(A)\otimes C(L)\subseteq C(A\otimes L).$$
\end{proposition}
\begin{proof}
For any $a_i\in A, x_i\in\mathcal{H}(L), 1\leq i\leq n$, and any $f\in C(A)$ and $\varphi\in C(L),$
\begin{eqnarray}
&& (f\otimes\varphi)[a_1\otimes x_1, \dots, a_n\otimes x_n]=(f\otimes \varphi)(a_1\dots a_n)\otimes[x_1, \dots, x_n]\nonumber\\
&&=f(a_1\dots a_n)\otimes\varphi[x_1, \dots, x_n]\nonumber\\
&&=\varepsilon(\varphi, X_i)a_1\dots f(a_i)\dots a_n\otimes [\alpha^{k}(x_1), \dots, \varphi(x_i), \dots, \alpha^{k}(x_n)]\nonumber\\
&&=\varepsilon(\varphi, X_i)[a_1\otimes \alpha^{k}(x_1)\dots f(a_i)\otimes\varphi(x_i), \dots, a_n\otimes \alpha^{k}(x_n)]\nonumber\\
&&=\varepsilon(\varphi, X_i)[\alpha'^{k}(a_1\otimes x_1)\dots (f\otimes\varphi)( a_i\otimes x_i), \dots, \alpha'^{k}(a_n\otimes x_n)]\nonumber.
\end{eqnarray}
Therefore, $f\otimes\varphi\in C(A\otimes L)$.
 \end{proof}


\section*{Acknowlegments} Dr. Ibrahima Bakayoko is grateful to the research environment in Mathematics and
Applied Mathematics MAM, the Division of Applied Mathematics of the School of Education, Culture and Communication at M{\"a}lardalen University for hospitality and an excellent and inspiring environment for research and research education and cooperation in Mathematics during his visit, in the framework of expanding research and research education capacity and cooperation development in Africa and impact of the programs in Mathematics between Sweden and countries in Africa supported by Swedish International Development Agency (Sida) and International Program in Mathematical Sciences (IPMS).


\begin{thebibliography}{99}
\bibitem{ChapBaSilnhomliecolor:Abramov}
Abramov, V.: Super 3-Lie algebras induced by super Lie algebras. Adv. Appl. Clifford Algebr. 27, no. 1, 9-16 (2017)

\bibitem{ChapBaSilnhomliecolor:AizawaSaito}
Aizawa, N., Sato, H.: $q$-deformation of the Virasoro algebra with central extension, Phys. Lett. B \textbf{256}, 185-190 (1991) (Hiroshima University preprint, preprint HUPD-9012 (1990))

\bibitem{ChapBaSilnhomliecolor:homdeformation}
Ammar, F., Ejbehi, Z., Makhlouf, A.: Cohomology and deformations of Hom-algebras,  J. Lie Theory \textbf{21}, no. 4, 813–836 (2011)

\bibitem{ChapBaSilnhomliecolor:naryhomrep}
Ammar, F., Mabrouk, S., Makhlouf, A.: Representation and cohomology of $n$-ary multiplicative Hom-Nambu-Lie algebras,  J. Geom. Phys. \textbf{61}, 1898–1913 (2011)

\bibitem{ChapBaSilnhomliecolor:envelopalgcolhomLiealg}
Armakan, A., Silvestrov, S., Farhangdoost, M.: Enveloping algebras of color hom-Lie algebras, Turk. J. Math. \textbf{43}, 316-339 (2019)
doi:10.3906/mat-1808-96. arXiv:1709.06164 [math.QA]

\bibitem{ChapBaSilnhomliecolor:exthomLiecoloralg}
Armakan, A., Silvestrov, S., Farhangdoost, M.: Extensions of hom-Lie color algebras, To appear in Georgian Mathematical Journal, doi:10.1515/gmj-2019-2033.
arXiv:1709.08620 [math.QA]

\bibitem{ChapBaSilnhomliecolor:akms:ternary}
Arnlind, J., Kitouni, A., Makhlouf, A., Silvestrov, S.:
Structure and Cohomology of $3$-Lie algebras induced by Lie algebras, in Algebra, Geometry and Mathematical Physics, Springer proceedings in Mathematics and $\&$ Statistics, vol \textbf{85} (2014)

\bibitem{ChapBaSilnhomliecolor:ams:ternary}
Arnlind, J., Makhlouf, A., Silvestrov, S.:
Ternary Hom-Nambu-Lie algebras induced by Hom-Lie algebras, J. Math. Phys. \textbf{51}, 043515, 11 pp. (2010)

\bibitem{ChapBaSilnhomliecolor:ams:n}
Arnlind, J., Makhlouf, A., Silvestrov, S.:
Construction of $n$-Lie algebras and $n$-ary Hom-Nambu-Lie algebras, J. Math. Phys. \textbf{52}, 123502, 13 pp. (2011)

\bibitem{ChapBaSilnhomliecolor:AtMaSi:GenNambuAlg}
Ataguema, H.,  Makhlouf, A., Silvestrov, S.: Generalization of n-ary Nambu
algebras and beyond. J. Math. Phys. 50, 083501 (2009) 

\bibitem{ChapBaSilnhomliecolor:almy:quantnambu}
Awata, H., Li, M., Minic, D., Yoneya, T.: On the quantization of Nambu brackets,
J. High Energy Phys. \textbf{2}, Paper 13, 17 pp. (2001)

\bibitem{ChapBaSilnhomliecolor:Bai:rlz3}
Bai, R., Bai, C., Wang, J.: Realizations of $3$-Lie algebras, Journal of Mathematical Physics \textbf{51}, 063505 (2010)

\bibitem{ChapBaSilnhomliecolor:Bai:n}
Bai, R., Wu, Y., Li, J., Zhou, H.: Constructing $(n+1)$-Lie algebras from $n$-Lie algebras, J. Phys. A \textbf{45}, no. 47 (2012)

\bibitem{ChapBaSilnhomliecolor:Bai:nLie:clas}
Bai, R., Song, G., Zhang, Y.: On classification of $n$-Lie algebras, Front. Math. China \textbf{6}, 581-606 (2011)

\bibitem{ChapBaSilnhomliecolor:Bai:nLie:claschar2}
Bai, R., Wang, X., Xiao, W., An, H.:
The structure of low dimensional $n$-Lie algebras over the field of characteristic $2$, Linear Algebra Appl. \textbf{428} (8–9), 1912-1920 (2008)

\bibitem{ChapBaSilnhomliecolor:RL}
Bai, R., Chen, L., Meng, D.: The Frattini subalgebra of n-Lie algebras, Acta Math. Sinica, English Series, 23 (5) 847-856 (2007)

\bibitem{ChapBaSilnhomliecolor:RM1}
Bai, R., Meng, D.: The central extension of n-Lie algebras, Chinese Ann. Math. 27 (4) 491-502 (2006)

\bibitem{ChapBaSilnhomliecolor:RM2}
Bai, R., Meng, D.: The centroid of n-Lie algebras, Algebras Groups Geom. 25 (2) 29-38 (2004)

\bibitem{ChapBaSilnhomliecolor:RB}
Bai, R., Zhang, Z., Li, H., Shi, H.: The inner derivation algebras of (n+1)-dimensional n-Lie algebras, Comm. Algebra, 28 (6) 2927-2934 (2000)

\bibitem{ChapBaSilnhomliecolor:IBLaplacehomLiequasibialg}
Bakayoko, I.: Laplacian of Hom-Lie quasi-bialgebras, International Journal of Algebra, \textbf{8} (15), 713-727 (2014)

\bibitem{ChapBaSilnhomliecolor:IBLmodcomodhomLiequasibialg}
Bakayoko, I.: L-modules, L-comodules and Hom-Lie quasi-bialgebras, African Diaspora Journal of Mathematics, Vol 17 49-64 (2014)

\bibitem{ChapBaSilnhomliecolor:Hombiliform}
Benayadi, S., Makhlouf, A.: Hom-Lie algebras with symmetric invariant nondegenerate bilinear forms, J. Geom. Phys. \textbf{76}, 38–60 (2014)

\bibitem{ChapBaSilnhomliecolor:JM}
Casas, J. M., Loday, J.-L., Pirashvili, T.: Leibniz $n$-algebras, Forum Math. 14, 189-207 (2002)

\bibitem{ChapBaSilnhomliecolor:ChaiElinPop}
Chaichian, M., Ellinas, D., Popowicz, Z.: Quantum conformal algebra with central extension, Phys. Lett. B \textbf{248}, 95-99 (1990)

\bibitem{ChapBaSilnhomliecolor:ChaiIsLukPopPresn}
Chaichian, M., Isaev, A. P., Lukierski, J., Popowic, Z., Pre\v{s}najder, P.: $q$-deformations of Virasoro algebra and conformal dimensions, Phys. Lett. B \textbf{262} (1), 32-38 (1991)

\bibitem{ChapBaSilnhomliecolor:ChaiKuLuk}
Chaichian, M., Kulish, P., Lukierski, J.: $q$-deformed Jacobi identity, $q$-oscillators and $q$-deformed infinite-dimensional algebras, Phys. Lett. B \textbf{237}, 401-406 (1990)

\bibitem{ChapBaSilnhomliecolor:ChaiPopPres}
Chaichian, M., Popowicz, Z., Pre\v{s}najder, P.: $q$-Virasoro algebra and its relation to the $q$-deformed KdV system, Phys. Lett. B \textbf{249}, 63-65 (1990)

\bibitem{ChapBaSilnhomliecolor:CM}
Chen, L., Ma, Y., Ni, L.: Generalized Derivations of Lie color algebras, Results Math., \textbf{63} (3-4), 923-936 (2013)

\bibitem{ChapBaSilnhomliecolor:CurtrZachos1}
Curtright, T. L., Zachos, C. K.: Deforming maps for quantum algebras, Phys. Lett. B \textbf{243}, 237-244 (1990)

\bibitem{ChapBaSilnhomliecolor:DamKu}
Damaskinsky, E. V., Kulish, P. P.: Deformed oscillators and their applications (in Russian), Zap. Nauch. Semin. LOMI 189, 37-74 (1991) (Engl. transl. in J. Sov. Math., 62, 2963-2986 (1992))

\bibitem{ChapBaSilnhomliecolor:DaskaloyannisGendefVir}
Daskaloyannis, C.: Generalized deformed Virasoro algebras, Modern Phys. Lett. A \textbf{7} no. 9, 809-816 (1992)

\bibitem{ChapBaSilnhomliecolor:DalTakh}
Daletskii, Y. L., Takhtajan, L. A.: Leibniz and Lie algebra structures for Nambu algebra, Lett. Math. Phys. \textbf{39}, 127-141 (1997)

\bibitem{ChapBaSilnhomliecolor:aip:review}
De~Azc{\'a}rraga, J.~A.,  Izquierdo, J.~M.:  $n$-ary algebras: a review with applications, J. Phys. A: Math. Theor. 43, 293001 (2010)

\bibitem{ChapBaSilnhomliecolor:Filippov:nLie}
Filippov, V. T.: $n$-Lie algebras, Siberian Math. J. \textbf{26}, 879-891 (1985)
(Transl. from Russian: Sib. Mat. Zh., 26 126-140 (1985))

\bibitem{ChapBaSilnhomliecolor:Gautheron:Rem}
Gautheron, Ph.: Some remarks concerning Nambu mechanics, Lett. Math. Phys. \textbf{37}, no. 1, 103-116 (1996)

\bibitem{ChapBaSilnhomliecolor:HLS}
Hartwig, J. T., Larsson, D., Silvestrov, S. D.:
Deformations of Lie algebras using $\sigma-$derivations, J. Algebra \textbf{295},  314-361 (2006)
(Preprint in Mathematical Sciences 2003:32, LUTFMA-5036-2003, Centre for Mathematical Sciences, Department of Mathematics, Lund Institute of Technology, 52 pp. (2003))

\bibitem{ChapBaSilnhomliecolor:HelSilbookqHeis}
Hellstr{\"o}m, L., Silvestrov, S. D.: Commuting Elements in $q$-Deformed Heisenberg Algebras, World Scientific, Singapore, 256 pp (1985) (ISBN: 981-02-4403-7)

\bibitem{ChapBaSilnhomliecolor:Hu}
Hu, N.: $q$-Witt algebras, $q$-Lie algebras, $q$-holomorph structure and representations,  Algebra Colloq. \textbf{6}, no. 1, 51-70 (1999)

\bibitem{ChapBaSilnhomliecolor:Kassel92}
Kassel, C.: Cyclic homology of differential operators, the virasoro algebra and a $q$-analogue, Comm. Math. Phys. 146 (2), 343-356 (1992)

\bibitem{ChapBaSilnhomliecolor:Kasymov:nLie}
Kasymov, Sh. M.: Theory of $n$-Lie algebras, Algebra and Logic. \textbf{26}, 155-166 (1987)
(Transl. from Russian: Algebra i Logika, Vol. 26, No. 3, pp. 277–297, (1987))

\bibitem{ChapBaSilnhomliecolor:KI1}
Kaygorodov, I.: On $\delta$-Derivations of n-ary algebras, Izvestiya: Mathematics, \textbf{76} (5) 1150-1162 (2012)

\bibitem{ChapBaSilnhomliecolor:KI2}
Kaygorodov, I.: $(n + 1)$-Ary derivations of simple n-ary algebras, Algebra and Logic, \textbf{50} (5) 470-471 (2011)

\bibitem{ChapBaSilnhomliecolor:KI3}
Kaygorodov, I.: $(n + 1)$-Ary derivations of semisimple Filippov algebras, Math. Notes, \textbf{96} (2) 208-216 (2014)

\bibitem{ChapBaSilnhomliecolor:KP}
Kaygorodov, I., Popov, Y.: Generalized derivations of (color) $n$-ary algebras, Linear and multilinear algebra, \textbf{64} (6) (2016)

\bibitem{ChapBaSilnhomliecolor:km:nary}
Kitouni, A., Makhlouf, A.: On structure and central extensions of $(n+1)$-Lie algebras induced by $n$-Lie algebras, arXiv:1405.5930 (2014)

\bibitem{ChapBaSilnhomliecolor:kms:nhominduced}
Kitouni, A., Makhlouf, A., Silvestrov, S.: On $(n+1)$-Hom-Lie algebras induced by $n$-Hom-Lie algebras Georgian Math. J. \textbf{23} no. 1, 75-95 (2016)

\bibitem{ChapBaSilnhomliecolor:LarssonSigSilvJGLTA2008}
Larsson, D., Sigurdsson, G., Silvestrov, S. D.: Quasi-Lie deformations on the algebra $\mathbb{F}[t]/(t^N)$,
J. Gen. Lie Theory Appl. \textbf{2}, 201-205 (2008)

\bibitem{ChapBaSilnhomliecolor:LS1}
Larsson, D., Silvestrov, S. D.: Quasi-Hom-Lie algebras, central extensions and $2$-cocycle-like identities, J. Algebra \textbf{288}, 321-344 (2005) (Preprints in Mathematical Sciences 2004:3, LUTFMA-5038-2004, Centre for Mathematical Sciences, Department of Mathematics, Lund Institute of Technology, Lund University (2004)).

\bibitem{ChapBaSilnhomliecolor:LS2}
Larsson, D., Silvestrov, S. D.: Quasi-Lie algebras. In "Noncommutative Geometry and Representation Theory in Mathematical Physics". Contemp. Math., 391, Amer. Math. Soc., Providence, RI, 241-248 (2005) (Preprints in Mathematical Sciences 2004:30, LUTFMA-5049-2004, Centre for Mathematical Sciences, Department of Mathematics, Lund Institute of Technology, Lund University (2004))

\bibitem{ChapBaSilnhomliecolor:LSGradedquasiLiealg}
Larsson, D., Silvestrov, S. D.: Graded quasi-Lie agebras, Czechoslovak J. Phys. \textbf{55}, 1473-1478 (2005)

\bibitem{ChapBaSilnhomliecolor:LS3}
Larsson, D., Silvestrov, S. D.: Quasi-deformations of $sl_2(\mathbb{F})$ using twisted derivations, Comm. in Algebra \textbf{35}, 4303-4318 (2007)

\bibitem{ChapBaSilnhomliecolor:LW}
Ling, W. X.:  On the structure of $n$-Lie algebras, PhD Thesis, University-GHS-Siegen, Siegen (1993)

\bibitem{ChapBaSilnhomliecolor:LiuKQuantumCentExt}
Liu, K. Q.: Quantum central extensions, C. R. Math. Rep. Acad. Sci. Canada \textbf{13} (4), 135-140 (1991)

\bibitem{ChapBaSilnhomliecolor:LiuKQCharQuantWittAlg}
Liu, K. Q.: Characterizations of the Quantum Witt Algebra, Lett. Math. Phys. \textbf{24} (4), 257-265 (1992)

\bibitem{ChapBaSilnhomliecolor:LiuKQPhDthesis}
Liu, K. Q.: The Quantum Witt Algebra and Quantization of Some Modules over Witt Algebra, PhD Thesis, Department of Mathematics, University of Alberta, Edmonton, Canada (1992)

\bibitem{ChapBaSilnhomliecolor:LodayPirash}
Loday, J-L.,  Pirashvili T.: Universal enveloping algebras of Leibniz algebras and (co)homology. Math. Ann. \textbf{296} no. 1, 139-158 (1993)

\bibitem{ChapBaSilnhomliecolor:ms:homstructure}
Makhlouf, A., Silvestrov, S. D.: Hom-algebra structures.
J. Gen. Lie Theory Appl. Vol \textbf{2} (2), 51-64 (2008)
(Preprints in Mathematical Sciences  2006:10, LUTFMA-5074-2006, Centre for Mathematical Sciences, Department of Mathematics, Lund Institute of Technology, Lund University (2006))

\bibitem{ChapBaSilnhomliecolor:Nambu:GenHD}
Nambu, Y.: Generalized Hamiltonian dynamics, Phys. Rev. D (3) \textbf{7}, 2405-2412 (1973)

\bibitem{ChapBaSilnhomliecolor:RichardSilvestrovJA2008}
Richard, L., Silvestrov, S. D.: Quasi-Lie structure of $\sigma$-derivations of $\mathbb{C}[t^{\pm1}]$, J. Algebra  319,  no. 3, 1285-1304 (2008)

\bibitem{ChapBaSilnhomliecolor:shenghomrep}
Sheng, Y.: Representation of Hom-Lie algebras, Algebr. Reprensent. Theory \textbf{15}, no. 6, 1081-1098 (2012)

\bibitem{ChapBaSilnhomliecolor:RM}
Rotkiewicz, M.: Cohomology ring of n-Lie algebras, Extracta Math. \textbf{20}, 219-232 (2005)

\bibitem{ChapBaSilnhomliecolor:SigSilvGLTbdSpringer2009}
Sigurdsson, G., Silvestrov, S.: Lie color and Hom-Lie algebras of Witt type and their central extensions, In "Generalized Lie theory in mathematics, physics and beyond", Springer, Berlin, 247-255 (2009)

\bibitem{ChapBaSilnhomliecolor:Czech:witt}
Sigurdsson, G., Silvestrov, S.: Graded quasi-Lie algebras of Witt type, Czech. J. Phys. 56: 1287-1291 (2006)

\bibitem{ChapBaSilnhomliecolor:Takhtajan:foundgenNambuMech}
Takhtajan, L. A.: On foundation of the generalized Nambu mechanics, Comm. Math. Phys., \textbf{160}, no. 2, 295-315 (1994)

\bibitem{ChapBaSilnhomliecolor:Takhtajan:cohomology}
Takhtajan, L. A.: Higher order analog of Chevalley-Eilenberg complex and deformation theory of $n$-gebras, St. Petersburg Math. J. \textbf{6} no. 2, 429-438 (1995)

\bibitem{ChapBaSilnhomliecolor:YauGenCom}
Yau, D.: A Hom-associative analogue of Hom-Nambu algebras, arXiv: 1005.2373 [math.RA] (2010)

\bibitem{ChapBaSilnhomliecolor:YauHomEnv}
Yau, D.: Enveloping algebras of Hom-Lie algebras, J. Gen. Lie Theory Appl. \textbf{2}, no. 2, 95-108 (2008)

\bibitem{ChapBaSilnhomliecolor:YauHomHom}
Yau, D.: Hom-algebras and homology, Journal of Lie Theory 19, No. 2, 409-421 (2009)

\bibitem{ChapBaSilnhomliecolor:YauHomNambuLie}
Yau, D.: On $n$-ary Hom-Nambu and Hom-Nambu-Lie algebras,  J. Geom. Phys. \textbf{62}, 506-522 (2012)

\end{thebibliography}
\end{document}